\documentclass[a4paper,11pt]{article}

\usepackage{microtype}
\usepackage[plainpages=false,colorlinks=true,
            linkcolor=black,urlcolor=black,citecolor=black]{hyperref}
\usepackage[margin=2.5cm]{geometry}
\usepackage{amsmath,amssymb,amsthm, amsfonts,mathrsfs}
\usepackage{color}
\usepackage{booktabs}
\usepackage{graphicx}
\usepackage{array, tabularx}
\usepackage{paralist} 
\usepackage{verbatim} 
\usepackage{subfig} 
\usepackage{xcolor,marginnote,enumitem}
\usepackage[numbers,sort]{natbib}
\DeclareMathOperator*{\argmax}{arg\,max}
\newcommand{\mO}{\mathcal{O}}

\newcommand\smallO{
  \mathchoice
    {{\scriptstyle\mathcal{O}}}
    {{\scriptstyle\mathcal{O}}}
    {{\scriptscriptstyle\mathcal{O}}}
    {\scalebox{.7}{$\scriptscriptstyle\mathcal{O}$}}
  }

\newtheorem{lemma}{Lemma}[section]
\newtheorem{theorem}{Theorem}[section]

\newtheorem{corollary}{Corollary}[section]

\theoremstyle{remark}

\theoremstyle{definition}

\numberwithin{equation}{section}
\DeclareMathOperator{\curl}{curl}
\DeclareMathOperator{\Div}{div}
\DeclareMathOperator{\grad}{grad}
\title{On an inverse elastic wave imaging scheme for \\nearly incompressible materials}

\author{\and Jingzhi Li \thanks{Department of Mathematics, Southern University of Science and
Technology, Shenzhen, P. R. China. ({\tt li.jz@sustc.edu.cn})} \and Hongyu Liu \thanks
{Department of Mathematics, Hong Kong Baptist University, Kowloon Tong, Hong Kong SAR, P. R. China.
({\tt hongyu.liuip@gmail.com; hongyuliu@hkbu.edu.hk})} \and Hongpeng Sun
\thanks {Institute for Mathematical Sciences, Renmin University of China, Beijing, P. R. China.
({\tt hpsun@amss.ac.cn})} }

\date{}
\begin{document}

\maketitle

\begin{abstract}

This paper is devoted to the algorithmic development of inverse elastic scattering problems.
We focus on reconstructing the locations and shapes of elastic scatterers with known dictionary data for the nearly incompressible materials. The scatterers include non-penetrable rigid obstacles and penetrable mediums, and we use time-harmonic elastic point signals as the incident input waves. The scattered waves are collected in a relatively small backscattering aperture on a bounded surface. A two-stage  algorithm is proposed for the reconstruction and only two incident waves of different wavenumbers are required. The unknown scatterer is first approximately located by using the measured data at a small wavenumber, and then the shape of the scatterer is determined by the computed location of the scatterer together with the measured data at a regular wavenumber. The corresponding mathematical principle with rigorous analysis is presented. Numerical tests illustrate the effectiveness and efficiency of  the proposed method.

\medskip

\medskip



\end{abstract}

\section{Introduction}\label{sec:intro}

Inverse scattering associated with acoustic, electromagnetic and elastic waves are important for various applications
including sonar and radar imaging, geophysical exploration, seismology, medical imaging and remote sensing; see \cite{AA1, AA2, AA3, CK, Kir1} and the references therein.
Inverse scattering problems are concerned with the recovery of unknown scatterers by wave probing. To that end, one sends an incident detecting wave to probe the scatterers, and then measures the scattered wave data away from the scatterers. By using the measurement data, one can infer knowledge about the unknown scatterers including locations, shapes or material properties of the scatterers.

In this paper, we mainly consider the inverse elastic wave imaging. In order to uniquely reconstruct the unknown obstacles or medium, one usually needs scattered data of full aperture for incident elastic plane waves of all directions at least theoretically; see \cite{PH} for the obstacles case and \cite{PH0} for the medium case with extra scattering data with multiple frequencies within some positive real interval. However, this causes a lot of challenges for practical applications, and particularly is not efficient for real-time applications. One reason is that no priori information is used. Inspired by the modern technology including machine learning, several imaging models and efficient numerical schemes with a priori information from dictionary data are developed \cite{AA1,AA2,AA3,LWY}. In \cite{AA1}, an efficient imaging procedure was developed for target identification based on the dictionary matching with precomputed generalized polarization tensors, and in \cite{LWY}, a
fast gesture computing scheme with acoustic wave is developed with precomputed scattering data.

To motivate the current study, we briefly discuss the ideas of the design in \cite{LWY}. Time-harmonic point wave signals are first emitted, and one then collects the scattering wave data within a relatively small backscattering aperture. The reconstruction process is divided into two steps. First, a low-frequency wave signal is emitted and one then uses the collected scattering data to determine the location of the scatterer. Second, a regular-frequency (compared to the size of the scatterers) wave signal is emitted, and one then uses the collected scattering data to determine the shape of the scatterer. Besides, the numerical implementation is totally ``direct" without any inversions or iterations, and hence it is very fast and robust.

Now we turn our focus to the inverse elastic problem. The nearly incompressible material is widely studied in engineering \cite{MDR,SMG} and mathematical community \cite{BM, SBC}.
For homogeneous deformations, the Lam\'{e} coefficients $\lambda$ and $\mu$ have the following descriptions \cite{LAU,SBC},
\begin{equation}\label{eq:strain:stree:rela}
\lambda  = \frac{\nu E}{(1+\nu)(1-2\nu)}, \quad \mu = \frac{E}{2(1+\nu)},
\end{equation}
where $E$ is the \emph{Young's modulus} and $\nu$ is the \emph{Poisson's ratio}. By \eqref{eq:strain:stree:rela}, we conclude the following limiting properties,
\begin{equation}\label{eq:nearly:in:para}
  \lambda \rightarrow +\infty \quad \text{while} \quad \nu \rightarrow (1/2)-,
\end{equation}
where this kind of material with \eqref{eq:nearly:in:para} is
commonly referred to as \emph{nearly incompressible material} \cite{SBC}. By \eqref{eq:nearly:in:para}, one can also use $\lambda/\mu \gg 1$ or $\lambda \gg \mu$ to characterize this property because of the boundedness of $\mu$ \cite{BLP}. The nearly incompressible materials are frequently seen in real applications including polymer materials, rubber and biologic tissues \cite{YT, MRR, MDR, KV}. The numerical computation of static elasticity or elastic wave propagation in the nearly incompressible materials is also very challenging and important, and various finite element or spectral element methods are developed in the last several decades for overcoming the computational difficulty \cite{SBC, BLP,BM, SMG}. In addition, there is also some related study on the inverse problems for engineering applications \cite{KV}.

In this article, we aim to further develop the imaging technique with a prior dictionary data in a practical setting of using elastic waves. There are several challenges we are confronted with the design of the reconstruction schemes. First, different from \cite{HLLS}, full aperture scattering data is usually neither practical nor feasible in assorted applications for nearly incompressible materials. The measurement information is very limited  in our numerical scheme, and the only available information is the backscattering data in a small aperture associated with a few time harmonic signals. Second, the whole reconstruction process should be completed in a timely manner. Third, elastic point signal includes the mixture of the fundamental solutions of two different wave numbers, and is more complicated than the acoustic or electromagnetic point signal which has only one wave number. Moreover, it is difficult to separate the shear wave and the pressure wave from the finite-aperture scattering data with different wave numbers, and meanwhile the separation is very critical for the ``direct" imaging method developed in \cite{LWY}.

Nonetheless, we manage to overcome the aforementioned challenges in our study by incorporating the key dictionary ingredient as mentioned before. We assume the shape of the unknown scatterers are all from a \emph{dictionary} that is known a priori, which is reasonable since the shapes of various scatterers can be collected by experience and learning.
A few dictionary techniques have been developed for inverse acoustic or electromagnetic scattering problems; see \cite{AA1, AA2, AA3, LWY} and the references therein. The critical ingredients of a dictionary method are the design of the appropriate dictionary class and the dictionary searching method. These are also the major technical contributions of the current article including the low frequency analysis for scattering by the nearly incompressible materials.
One needs first to determine the location of the scatterer. After that, one can use the dictionary matching algorithm to determine the specific shape. However, in the dictionary class, the scattering information of the admissible shapes should be independent of any location requirement. This challenge can be overcome by using the so-called translation relation if incident plane waves are used; see \cite{AA1, AA2, AA3, LWY, LLS}. However in the current design, elastic point signals with two different wave numbers are used and the scattering data are collected in a special manner. This requires some special technical analysis and treatments to determine the dominating shear or pressure wave in our study. Moreover, for timely dictionary matching, we propose a fast and robust ``direct" method with the dominating wave based on our detailed theoretical analysis.

The rest of the paper is organized as follows. In section \ref{sec:mathground}, we present the mathematical framework for our imaging scheme with elastic waves. In section \ref{sec:analysis:lowfre}, we give the necessary low frequency analysis for nearly incompressible materials with preparation for locating scatterers with a low frequency.  In section \ref{sec:two-stage}, we present the two-stage recognition algorithm based on the theoretical analysis. In section \ref{sec:num}, extensive numerical tests are conducted to verify the effectiveness and efficiency of the proposed algorithm. We conclude our study in Section~\ref{sect:conclusion} with some relevant discussion.

\section{Mathematical Framework}\label{sec:mathground} 

In this section, we present the general mathematical setting and fundamentals for the proposed reconstruction scheme.
The shape of the unknown scatterer is supposed to be a $C^2$ domain $\Omega$, which is assumed to have a connected complement $\Omega^c := \mathbb{R}^{3} \backslash \bar \Omega$. It is assumed that there exists a {\it dictionary} of $C^2$ domains, which can be calibrated beforehand, i.e.,
\begin{equation}\label{eq:dic:class}
\mathfrak{D} = \{ D_j\}_{j=1}^{N}, \quad N \in \mathbb{N}.
\end{equation}
Here each $D_j$ is simply connected and contains the origin, such that there exists a translation operator $F: \mathbb{R}^3 \rightarrow \mathbb{R}^3$,
\begin{equation}\label{eq:trasi:domain}
\Omega = F(D) = D + z: = \{ x+z; \ x \in D \}, \quad D  \in \mathfrak{D}, \quad z \in \mathbb{R}^3.
\end{equation}
Our reconstruction strategy with elastic waves is a typical inverse elastic scattering problem.  In the setup of the current study, the scatterer $\Omega$ or the dictionary domain $D_j$ shall be assumed to be a non-penetrable rigid obstacles or a penetrable medium scatterer, which  covers many practical scenarios of important applications \cite{AR,Kup}. The incident elastic point waves were located at a fixed position. With the incident waves, one then measures the scattered wave due to the unknown scatterer $\Omega$ on a measurement surface $\Gamma$ with multiple receivers. In our study, the measurement surface $\Gamma$ contains the location of the incident point waves.

Throughout the rest of the paper, we need the following two assumptions,
\begin{equation}\label{eq:assume:1}
\|D_j\|: = \max_{x \in D} |x| \simeq 1, \quad 1 \leq j \leq N,
\end{equation}
and
\begin{equation}\label{eq:assume:2}
|z| \gg 1,
\end{equation}
where $z$ is the location of $\Omega$ as in \eqref{eq:trasi:domain}. Assumption \eqref{eq:assume:1} means that the size of the scatterer $\Omega$ can be calibrated such that the regular frequency scale is characterized as $\frac{2 \pi}{k_s} \simeq \|\Omega\|$ and the low frequency of the elastic waves is characterized as $\frac{2 \pi}{k_s} \gg \|\Omega\|$, where $k_s\in\mathbb{R}_+$ signifies the wave number of the shear waves. This is practically feasible, since the frequency band of the elastic waves is of a wide range \cite{AR}. Throughout the rest of the paper, for exposition convenience, we always assume that $\rho\equiv 1$.


Actually, we would like to emphasize that we only need to collect the scattered field of $\Omega$ in a small aperture scattered by the elastic point source, which is convenient and quite practical. We also point out that this condition \eqref{eq:assume:2} is  mainly required for the theoretical justification of the proposed algorithm in what follows. Indeed, in our numerical tests, the proposed reconstruction algorithm works effectively and efficiently, as long as the scatterer $\Omega$ is located away from the point sources of a reasonable distance.

For the setup described above, we next introduce the direct elastic wave scattering. The displacement of a time-harmonic elastic wave is governed by the following Navier's equations,
\begin{equation}\label{eq:elastic:wave}
\mu \Delta u(x) +(\lambda + \mu) \grad \Div u(x) + \omega^2 \rho  u(x)  = 0,
\end{equation}
where $\lambda$ and $\mu$ are Lam\'{e} coefficients, $\rho$ the density, and $\omega$ the frequency. In the sequel, we take one scatterer $\Omega$ for example in the analysis. We mainly consider two cases, i.e., $\Omega$ is a non-penetrable rigid obstacle or  $\Omega$ is a penetrable medium.

For the rigid obstacle case, the elastic wave scattering is governed by the following homogeneous PDE system in $\mathbb{R}^3$, i.e., to find $u_{\omega,\Omega}^{s} \in H^{1}_{loc}(\Omega^c)$, such that
\begin{equation}\label{eq:scat:pec}
\begin{cases}
&\mu \Delta u_{\omega,\Omega}^s(x) +(\lambda + \mu) \grad \Div u_{\omega,\Omega}^s(x) + \omega^2 u_{\omega,\Omega}^s(x)  = 0,
 \quad x \in \Omega^c,  \\
& u_{\omega,\Omega}^s(x) + u_{\omega,p,ep}^i(x,y) = 0, \quad x \in \partial \Omega, \quad y \in \Omega^{c},  \\
&\displaystyle{\lim_{|x| \rightarrow \infty}|x|( \curl \curl u_{\omega,\Omega}^s \times \hat{x} -i k_{s}u_{\omega,\Omega}^s) = 0}, \\
 & \displaystyle{\lim_{|x| \rightarrow \infty}|x|( \hat{x}\cdot \grad \Div u_{\omega,\Omega}^s -ik_p \Div u_{\omega,\Omega}^s) = 0},
\end{cases}
\end{equation}
where the last two equations are known as the Kupradze radiation conditions \cite{Kup}. Moreover, $u_{\omega,p,ep}^i$ the elastic incident point source wave with a polarization $p\in\mathbb{R}^3$ and a source position $y\in\mathbb{R}^3$ \cite{PHhab},
\begin{equation}\label{eq:point:source}
u_{\omega,p,ep}^{i}(x,y)=\Gamma(x,y)p, \quad y \neq x,
\end{equation}
where $\Gamma(x,y)$ is the fundamental solution of \eqref{eq:elastic:wave} with $\rho \equiv 1$ (see \cite[Chap. 5]{PHhab}, \cite{PH}),
\begin{align}
  \Gamma(x,y)p :&= \frac{e^{ik_{s}|x-y|}}{4 \pi  \mu |x-y|}p + \frac{1}{\omega^2}\grad_{x} \Div_{x} \left[\frac{e^{ik_s|x-y|} -e^{ik_p|x-y|}}{4 \pi |x-y|}p\right] \label{eq:pointsource:1}\\
&=   \frac{1}{\omega^2}\curl_{x} \curl_{x} \left[\frac{e^{ik_{s}|x-y|}}{4 \pi |x-y|}p\right] - \frac{1}{\omega^2} \grad_{x} \Div_{x} \left[\frac{e^{ik_{p}|x-y|}}{4 \pi |x-y|}p\right],\label{eq:point}
\end{align}
with the wave numbers $k_s, k_p > 0 $ given by
\begin{equation*}
k_s = \frac{\omega}{\sqrt{\mu}},\ \quad k_p = \frac{\omega}{\sqrt{2\mu + \lambda}}.
\end{equation*}

For the medium case, denoting the mass density as $\rho_{\Omega}(x)$ and $n_{\Omega} =1 -\rho_{\Omega}$, we assume $n_{\Omega}$ has a compact support, and define
$\Omega: = \{x \in \mathbb{R}^3: n_{\Omega}(x) \neq 0\}$ as the non-homogeneous medium.
The elastic medium scattering problem of an inhomogeneous medium $\Omega$ with the incident elastic point source reads as follows:
to find $u_{\omega,\Omega}^s \in H^{1}_{loc}(\mathbb{R}^3)$, such that
\begin{equation}\label{eq:scat:medium}
\begin{cases}
&\mu \Delta u_{\omega,\Omega}(x) +(\lambda + \mu) \grad \Div u_{\omega,\Omega}(x) + \omega^2  (1-n_{\Omega}) u_{\omega,\Omega}(x)  = 0, \quad x \in \mathbb{R}^3\backslash \{y\},  \\
& u_{\omega,\Omega}(x) = u_{\omega,\Omega}^s(x) + u_{\omega,p,ep}^i(x,y), \quad  x \in \mathbb{R}^3\backslash \{y\},  \\
& \displaystyle{\lim_{|x| \rightarrow \infty}|x|( \curl \curl u_{\omega,\Omega}^s \times \hat{x} -i k_{s}u_{\omega,\Omega}^s) = 0},  \\
& \displaystyle{\lim_{|x| \rightarrow \infty}|x|( \hat{x}\cdot \grad \Div u_{\omega,\Omega}^s -ik_p \Div u_{\omega,\Omega}^s) = 0}.
\end{cases}
\end{equation}
The well-posedness of the direct scattering problem \eqref{eq:scat:pec} and \eqref{eq:scat:medium} are known \cite{PH, PH0, PHhab}.
The main focus of this paper is the following inverse problem:

\emph{Given the measured scattering field $u_{\omega,\Omega}^s(x)$ on a bounded surface $\Lambda$, to find the location $z$ and the shape of the scatterer $\Omega$ for the nearly incompressible material.}

 As introduced in Section \ref{sec:intro}, our reconstruction algorithm contains two stages. In the first stage, we locate the scatterer $\Omega$ by the measured scattering field $u_{\omega,\Omega}^s$ in \eqref{eq:scat:pec} or \eqref{eq:scat:medium} under low frequency on a bounded surface $\Lambda$.  To this end, we develop specially designed indicator functionals, which are originated from the low frequency analysis of the nearly incompressible materials. Once the location is found, we then collect the measured scattered field $u_{\omega,\Omega}^s$ on $\Lambda$ as in \eqref{eq:scat:pec} or \eqref{eq:scat:medium} under regular frequency to reconstruct the shape of $\Omega$. In each process, only one incident point source wave is needed. Additionally, for the shape determination, we shall benefit from the scattering field of $D$ scattered by the incident plane wave stored in the precomputed dictionary. With the priori information in the dictionary and the dictionary matching process, we can reconstruct the shape efficiently. In light of these, we next introduce the elastic wave scattering of $D$ due to an incident plane wave, which also includes both the non-penetrable rigid obstacle case and the penetrable medium case.

 We first introduce the elastic wave scattering of the rigid scatterer $D$, i.e., to find the radiating field $u^s(D,d,p;x) \in H^{1}_{loc}(D^c)$ satisfying the Kupradze radiation conditions, such that
\begin{align}
&\mu \Delta u^s(D,d,p;x) +(\lambda + \mu) \grad \Div u^s(D,d,p;x) + \omega^2 u^s(D,d,p;x)  = 0,
 \ \  x \in D^c = \mathbb{R}^3\backslash \bar D, \notag \\
& u^s(D,d,p;x) + u^i(x,d,p) = 0, \quad   x \in \partial D. \label{eq:scat:pec:plane}
\end{align}

For the elastic medium scattering of $D$, similar to the case of $\Omega$, we need to introduce $\rho_{D}(x)$ and $n_{D} =1 -\rho_{D}$. By the translation relation \eqref{eq:trasi:domain}, we have
\begin{equation}\label{eq:n:medium}
n_{\Omega}(y) = n_{D}(x),  \quad y = x+z, \quad x \in \mathbb{R}^3.
\end{equation}
The scattering of the inhomogeneous elastic medium $D$ by a plane incident wave is to find $u^s(D,d,p;x) \in H^{1}_{loc}(\mathbb{R}^3)$ with the Kupradze radiation conditions,  such that
\begin{align}
&\mu \Delta u(D,d,p;x) +(\lambda + \mu) \grad \Div u(D,d,p;x) + \omega^2  (1-n_{D}) u(D,d,p;x)  = 0, \quad x \in \mathbb{R}^3, \notag \\
& u(D,d,p;x)  = u^s(D,d,p;x) + u^i(x,d,p) = 0, \quad x \in \mathbb{R}^3.  \label{eq:scat:medium:plane}
\end{align}
$u^i(x,d,p)$ as in \eqref{eq:scat:pec:plane} and \eqref{eq:scat:medium:plane} is the elastic plane wave with the polarization $p \in \mathbb{R}^3$ and the travelling direction $d \in \mathbb{R}^3$, i.e., $u^{i}(x,d,p) =  u_{p}^{i}(x,d,p) + u_{s}^{i}(x,d,p)$,
\begin{subequations}\label{eq:plane}
\begin{align}
\quad u_{p}^{i}(x,d,p):&= -\frac{1}{\omega^2}\grad_{x} (\Div_{x}[ p e^{ik_p d \cdot x}]) =
\frac{k_{p}^2}{\omega^2} d\cdot p d e^{ik_px\cdot d}, \\
u_{s}^i(x,d,p):&= \frac{1}{\omega^2} \curl_x \curl_x[pe^{ik_sd\cdot x}]=  \frac{k_{s}^2}{\omega^2} (d\times p) \times d e^{ik_sx\cdot d}.
\end{align}
\end{subequations}

Moreover, the following asymptotic relations between the elastic point incident wave \eqref{eq:point} and elastic plane incident waves \eqref{eq:plane} are useful for our subsequent discussion. Denote $\hat{z} = z/|z|$ for $z\neq 0$. Here and also in what follows, we use $\mO(\cdot)$ or $\smallO(\cdot)$ to represent the usual asymptotic behavior of real scalar variables.
\begin{lemma}\label{lem:incident}
As $|z| \rightarrow \infty$, for large $\lambda \gg \mu$, we have
\begin{align}
 u_{\omega,p,ep}^{i}(x+z,y) &= \frac{e^{ik_p|z|-ik_p\hat{z} \cdot y}}{4 \pi |z|}u_{p}^{i}(x,\hat{z},p)+
\frac{e^{ik_s|z|-ik_s\hat{z} \cdot y}}{4 \pi |z|}u_{s}^{i}(x,\hat{z},p) + \mO(\frac{1}{|z|^2})\frac{k_p^2 + k_s^2}{\omega^2} \label{eq:asym:point1}\\
|u_{p}^{i}(x,\hat{z},p)|& = \mO(\frac{1}{\lambda}), \quad |u_{s}^{i}(x,\hat{z},p)| = \mO(\frac{1}{\mu}). \label{eq:asym:plane:incident}
\end{align}
\end{lemma}
\begin{proof}
First, we have
\begin{align}
\Gamma(x+z,y)p  &= \frac{1}{\omega^2}\curl_{x} \curl_{x} \left[\frac{e^{ik_{s}|x+z-y|}}{4 \pi |x+z-y|}p\right] - \frac{1}{\omega^2} \grad_{x} \Div_{x} \left[\frac{e^{ik_{p}|x+z-y|}}{4 \pi |x+z-y|}p\right] \notag \\
&=\frac{1}{\omega^2}\curl_{x} \curl_{x} \left[\frac{e^{ik_{s}|x-(y-z)|}}{4 \pi |x-(y-z)|}p\right] - \frac{1}{\omega^2} \grad_{x} \Div_{x} \left[\frac{e^{ik_{p}|x-(y-z)|}}{4 \pi |x-(y-z)|}p\right]. \label{point:green:expan}
\end{align}
For the first term of the RHS of \eqref{point:green:expan}, by direct calculation, we have 
\begin{equation}\label{up:first:expan}
\begin{aligned}
&\frac{1}{\omega^2}\curl_{x} \curl_{x} \left[\frac{e^{ik_{s}|x-(y-z)|}}{4 \pi |x-(y-z)|}p\right]  =
\frac{1}{\omega^2}(-\Delta + \nabla \Div )\left[\frac{e^{ik_{s}|x-(y-z)|}}{4 \pi |x-(y-z)|}p\right] \\
&=\frac{k_s^2}{\omega^2} \left[\Phi_{k_s}(x,y-z) \left(p-\frac{z-(y-x)}{|z-(y-x)|}\frac{[z-(y-x)]\cdot p}{|z-(y-x)|}\right) + \mathcal{O}(\frac{1}{|z|^2})\right] \\
& = \frac{k_s^2}{\omega^2} \left[ \frac{e^{ik_s|z|}}{4 \pi |z|}  e^{ik_s \hat{z} \cdot (x-y)}\{p - \hat{z} \hat{z} \cdot p\} + \mathcal{O}(\frac{1}{|z|^2})\right]  \\
& = \frac{k_s^2}{\omega^2} \left[\frac{e^{ik_s|z|-ik_s\hat{z} \cdot y}}{4 \pi |z|}  (\hat{z}\times p) \times \hat{z} e^{ik_sx\cdot \hat{z}}  + \mathcal{O}(\frac{1}{|z|^2})\right].
\end{aligned}
\end{equation}
For the second term of the RHS of \eqref{point:green:expan}, by direct calculation, we can obtain
\begin{equation}\label{up:second:expan}
\begin{aligned}
- \frac{1}{\omega^2} \grad_{x} \Div_{x} \left[\frac{e^{ik_{p}|x-(y-z)|}}{4 \pi |x-(y-z)|}p\right]
& = \frac{1}{\omega^2}(-\Delta_x - \curl_{x} \curl_{x})\left[\frac{e^{ik_{p}|x-(y-z)|}}{4 \pi |x-(y-z)|}p\right]\\
&=\frac{k_{p}^2}{\omega^2} \left[\frac{e^{ik_p|z|-ik_p\hat{z} \cdot y}}{4 \pi |z|}  \hat{z} \hat{z}\cdot p e^{ik_px\cdot \hat{z}}+\mathcal{O}(\frac{1}{|z|^2})\right].
\end{aligned}
\end{equation}
By \eqref{up:first:expan} and \eqref{up:second:expan}, compared with \eqref{eq:plane}, we have \eqref{eq:asym:point1}.
Combining \eqref{eq:point:source}, \eqref{eq:plane} and the assumption that $\lambda \gg \mu$, we finally arrive at \eqref{eq:asym:plane:incident}.
\end{proof}
%
Henceforth, we use $\nu$ to signify the exterior unit normal of the domain concerned.  For a vector $u \in C^{1}(D^c)^3$ (or $u \in C^{1}(\Omega^c)^3$) where $D^c = \mathbb{R}^3 \backslash \bar D$ (or $\Omega^c = \mathbb{R}^3 \backslash \bar \Omega$), we define for $x \in \partial D$ (or $x \in \partial \Omega$ )
 \begin{equation}\label{eq:traction:elastic}
 [\mathcal{P}u](x): = (\alpha + \mu)(\nu\cdot \grad)u + \beta \nu \Div u + \alpha [\nu  \times \curl u(x)],
 \end{equation}
which is the traction vector at $x$, with $\alpha$ and $\beta$ defined as follows
\begin{equation}\label{eq:alpha:beta}
\alpha :=\frac{\mu(\lambda + \mu)}{\lambda+3\mu}, \quad \beta: = \frac{(\lambda+\mu)(\lambda+2\mu)}{\lambda+3 \mu}.
\end{equation}
For $y \in \partial D$, $x\in \mathbb{R}^3$, $x \neq y$, we define $\Pi(x,y) \in \mathbb{C}^{3\times3}$ by
\[
\Pi(x,y)^{T}P := \mathcal{P}_{y}(\Gamma(x,y)P),
\]
where the superscript $T$ denotes the transpose.
For $\phi(x) \in C(\partial D)^3$, we introduce the following single and double layers,
\[
(S_{\omega, D} \phi)(x) : = 2 \int_{\partial D} \Gamma(x,y) \phi(y)ds(y), \quad (K_{\omega, D}\phi)(x): = 2 \int_{\partial D} \Pi(x,y) \phi(y)ds(y).
\]
It is well-known \cite{PH} that  $\Pi(x,y)$ is weakly singular with specifically chosen $\alpha$, $\beta$ in \eqref{eq:alpha:beta}, and $K_{\omega, D}: C(\partial D)^3 \rightarrow C^{0, \alpha}(\partial D)^3$ is bounded.
For $\varphi(x) \in C(D)^3$, we introduce the volume potential
\[
(V_{\omega, D}\varphi)(x): = \int_{D}\Gamma(x,y)\varphi(y)dy.
\]
 In the sequel, we study the translation relation of the scattered field of $\Omega$ due to an incident point source $u_{\omega,p,ep}^i$ and the scattered field of the translated $D$ due to an incident plane wave $u^i(x,d,p)$. For expositional simplification and by normalization, we assume in the following that $\rho\equiv 1$ for the background space $\mathbb{R}^3\backslash \overline{\Omega}$.

\subsection{Elastic wave scattering of rigid obstacle}

In fact, for the scattered elastic waves of system \eqref{eq:scat:pec} and \eqref{eq:scat:pec:plane}, we have the following asymptotic relation under the translation condition \eqref{eq:trasi:domain}. For clarity, we assume that the point source in \eqref{eq:point:source} is located in $y_0$ instead of $y$ hereafter.
 \begin{theorem}\label{thm:pec}
For a fixed $\omega \in \mathbb{R}_{+}$, we have the following asymptotic expansions for the rigid obstacle scattering problem \eqref{eq:scat:pec},
 \begin{align}\label{eq:rr}
 u_{\omega,\Omega}^s(x)
 & =\frac{e^{ik_p|z|-ik_p\hat{z} \cdot y_0}}{4 \pi |z|}u_{p}^s(D,\hat{z},p;x-z)+ \frac{e^{ik_s|z|-ik_s\hat{z} \cdot y_0}}{4 \pi |z|}u_{s}^s(D,\hat{z},p;x-z)+ \mathcal{O}(|z|^{-2}),
 \end{align}
 for any fixed $x \in \Omega^c$ as $|z| \rightarrow \infty$ uniformly for all $\hat{z} \in \mathbb{S}^2$.
In \eqref{eq:rr}, $u_{p}^s(D,\hat{z},p;x)$ and $u_{s}^s(D,\hat{z},p;x)$ are the scattered fields of the rigid obstacle $D$ as in \eqref{eq:scat:pec:plane} corresponding to the incident plane waves
  $u_{p}^i(x,\hat{z},p)$ and $u_{s}^i(x,\hat{z},p)$ respectively.
 \end{theorem}
 \begin{proof}
 By \cite{PH}, $ u_{\omega,\Omega}^s(x) $ can be represented as the following combined potentials,
 \begin{equation}
 u_{\omega,\Omega}^s(x) =  2\int_{\partial \Omega}\Pi(x,y) \phi(y)ds(y) + 2i  \int_{\partial \Omega} \Gamma(x,y) \phi(y)ds(y),
 \end{equation}
 where $\phi(y)$ is a vectorial density on $\partial \Omega$.
 Here and also in what follows, we denote
 \[
 \Phi_{\kappa}(x,y) = e^{i \kappa  |x-y|}/({4 \pi|x-y|}), \quad \text{with} \ \kappa = k_s \ \text{or} \ \kappa = k_p.
 \]
By direct calculations, one can verify that for any $f(y) \in \mathbb{C}^3$,
  \begin{align*}
  \Delta_{x}[f(y)\Phi_{\kappa}(x,y)] &= \Delta_{x}\Phi_{\kappa}(x,y) f(y), \\
  \curl_{x} \curl_{x} [f(y)\Phi_{\kappa}(x,y)] & = (-\Delta_{x} + \nabla_{x} \Div_{x})[f(y)\Phi_{\kappa}(x,y)].
  \end{align*}
  It can be shown that $ \Phi_{\kappa}(x+z,y) = \Phi_{\kappa}(x,y-z)$ and
  \[
  \quad \frac{\partial \Phi_{\kappa}(\tilde x,y)}{\partial \tilde x_{j}}|_{\tilde{x} = x+z} =   \frac{\partial \Phi_{\kappa}( x,y-z)}{\partial x_{j}},\quad   \frac{\partial^2 \Phi_{\kappa}(\tilde x,y)}{\partial \tilde x_{j}\partial \tilde x_{i}}|_{\tilde{x} = x+z} =   \frac{\partial^2 \Phi_{\kappa}( x,y-z)}{\partial x_{j}\partial x_{j}}.
  \]
  Thus we have
  \begin{equation}\label{eq:curl:trans}
  \begin{aligned}
  \grad_{x}\Div_{x}[f(y)\Phi_{\kappa}(x,y)](x+z) &=   \grad_{x}\Div_{x}[f(y)\Phi_{\kappa}(x,y-z)], \\
  \curl_{x} \curl_{x} [f(y)\Phi_{\kappa}(x,y)](x+z) & = \curl_{x} \curl_{x} [f(y)\Phi_{\kappa}(x,y-z)] .
  \end{aligned}
\end{equation}

We are in a position to prove the theorem. First, considering $ u_{\omega,\Omega}^s(x+z)$ with $x\in \Omega^c$, we
  have
  \begin{align*}
  \curl_{x}[f(y)\Phi_{\kappa}(x,y)] &= \nabla_{x}\Phi_{\kappa}(x,y)\times f(y), \\
  \curl_{x} \curl_{x} [f(y)\Phi_{\kappa}(x,y)] & = (-\Delta_{x} + \nabla_{x} \Div_{x})[f(y)\Phi_{\kappa}(x,y)].
  \end{align*}
It can be shown that $ \Phi_{\kappa}(x+z,y) = \Phi_{\kappa}(x,y-z)$ and
  \[
  \quad \frac{\partial \Phi_{\kappa}(\tilde x,y)}{\partial \tilde x_{j}}|_{\tilde{x} = x+z} =   \frac{\partial \Phi_{\kappa}( x,y-z)}{\partial x_{j}},\quad   \frac{\partial^2 \Phi_{\kappa}(\tilde x,y)}{\partial \tilde x_{j}\partial \tilde x_{i}}|_{\tilde{x} = x+z} =   \frac{\partial^2 \Phi_{\kappa}( x,y-z)}{\partial x_{j}\partial x_{j}}.
  \]
Hence, we have
  \begin{equation}\label{eq:curlcurl:trans}
  \begin{aligned}
  \curl_{x}[f(y)\Phi_{\kappa}(x,y)](x+z) &=   \curl_{x}[f(y)\Phi_{\kappa}(x,y-z)], \\
  \curl_{x} \curl_{x} [f(y)\Phi_{\kappa}(x,y)](x+z) & = \curl_{x} \curl_{x} [f(y)\Phi_{\kappa}(x,y-z)] .
  \end{aligned}
\end{equation}
     By the jump relation of the vector potentials of $K_{\omega, \Omega}$ and $S_{\omega, \Omega}$ \cite{PH, PHhab}, we have
  \begin{equation}\label{eq:density:eq:Omega}
  \phi(x) + (K_{\omega,\Omega}\phi)(x) + i (S_{\omega,\Omega} \phi)(x) = -  u_{\omega, p,ep}^i(x,y_0), \quad x \in \partial \Omega, \quad y_0 \notin  \partial \bar \Omega
  \end{equation}
Then we shall write $  u_{\omega,\Omega}^s(x+z)$ by \eqref{eq:curl:trans} and \eqref{eq:curlcurl:trans} as follows,
  \begin{align}
  u_{\omega,\Omega}^s(x+z)
                           & =  2\int_{\partial \Omega}\Pi(x+z,y) \phi(y)ds(y) + 2i  \int_{\partial \Omega} \Gamma(x+z,y) \phi(y)ds(y),
     \label{eq:E:repre:Omega}
  \end{align}
   By changing variable with $y = z+t$ in \eqref{eq:E:repre:Omega} and the assumption \eqref{eq:trasi:domain}, we have
   \begin{align}
     u_{\omega,\Omega}^s(x+z)
      &=  2\int_{\partial D}\Pi(x,t) \phi(z+t)ds(t) + 2i  \int_{\partial D} \Gamma(x,t) \phi(z+t)ds(t), \notag \\
      & = [(K_{\omega,D}\phi(t+z))(x)+ i(S_{\omega,D}\phi(t+z))(x)].\label{eq:E:representaion:D}
   \end{align}
   In the sequel, we turn to solving the density $\phi(z+t)$ with $t \in \partial D$.
   Letting $x = t+z$ in \eqref{eq:density:eq:Omega} and denoting $\tilde \phi(y)|_{y\in\partial D} = \phi(t+z)|_{t \in \partial D}$, again by the change of variables and similar arguments as before, we have
   \begin{equation}\label{eq:density:sa}
     [(I + K_{\omega,D} + iS_{\omega,D})\tilde \phi](t+z)|_{\partial D} = - u_{\omega,p,ep}^i(t+z, y_0), \quad t \in \partial D.
   \end{equation}
   By \eqref{eq:density:sa}, we have
   \begin{equation}\label{eq:a:trans:D:a}
   \phi(t+z) = \tilde \phi(y) = (I + K_{\omega,D} + i S_{\omega,D})^{-1}(- u_{\omega,p, ep}^i(\cdot+z, y_0)), \quad t \in \partial D.
   \end{equation}
   Substituting $\phi(t+z)$ in \eqref{eq:a:trans:D:a} into \eqref{eq:E:representaion:D}, together with Lemma \ref{lem:incident}, we have for any $x \in \Omega^c$
   \begin{align}
     u_{\omega,\Omega}^s(x+z) &= [K_{\omega,D} + iS_{\omega,D}]  (I + K_{\omega,D} + i S_{\omega,D})^{-1}[-  u_{\omega,p,ep}^i(\cdot+z,y_0)], \notag \\
     &= \frac{e^{ik_p|z|-ik_p\hat{z} \cdot y_0}}{4 \pi |z|}  [K_{\omega,D} + iS_{\omega,D}]  (I + K_{\omega,D} + i S_{\omega,D})^{-1}[- u_{p}^{i}(t,\hat{z},p)+ \mathcal{O}(|z|^{-1})], \notag \\
     & \quad +\frac{e^{ik_s|z|-ik_s\hat{z} \cdot y_0}}{4 \pi |z|}  [K_{\omega,D} + iS_{\omega,D}]  (I + K_{\omega,D} + i S_{\omega,D})^{-1}[- u_{s}^{i}(t,\hat{z},p)+ \mathcal{O}(|z|^{-1})], \notag  \\
     &= \frac{e^{ik_p|z|-ik_p\hat{z} \cdot y_0}}{4 \pi |z|}u_{p}^s(D,\hat{z},p;x) + \frac{e^{ik_s|z|-ik_s\hat{z} \cdot y_0}}{4 \pi |z|}u_{s}^s(D,\hat{z},p;x)+ \mathcal{O}(|z|^{-2}). \label{eq:asym:rigid:twp}
   \end{align}
In \eqref{eq:asym:rigid:twp}, $u_{p}^s(D,\hat{z},p;x)$ and $u_{s}^s(D,\hat{z},p;x)$ are the scattering fields corresponding to the incident plane elastic wave $u_{p}^i(x, \hat{z}, p)$ and $u_{s}^i(x,\hat{z},p)$ respectively, i.e., $u_{p}^s(D,\hat{z},p;x)$ $= T_{\omega,D}(- u_{p}^{i}(t,\hat{z},p))(x)$ and $u_{s}^s(D,\hat{z},p;x)$ $= T_{\omega,D}(-u_{s}^{i}(t,\hat{z},p))(x)$ with $T_{\omega,D}$  defined as follows,
    \[
    T_{\omega,D} : C(\partial D)^3 \rightarrow C(\partial D)^3, \quad T_{\omega,D} = [K_{\omega,D} + iS_{\omega,D}]  (I + K_{\omega,D} + i S_{\omega,D})^{-1}.
    \]
Therefore, by change of variables, we have
   \begin{equation}
   u_{\omega,\Omega}^s(x) =  \frac{e^{ik_p|z|-ik_p\hat{z} \cdot y_0}}{4 \pi |z|}u_{p}^s(D,\hat{z},p;x-z) + \frac{e^{ik_s|z|-ik_s\hat{z} \cdot y_0}}{4 \pi |z|}u_{s}^s(D,\hat{z},p;x-z)+ \mathcal{O}(|z|^{-2}). \label{eq:asym:rigid:twoo}
   \end{equation}
   which completes the proof.

 \end{proof}



\subsection{Elastic scattering of inhomogeneous medium}

For the scattered elastic waves of system \eqref{eq:scat:medium} and \eqref{eq:scat:medium:plane} of the penetrable medium scattering, we have the following asymptotic relation under the translation condition \eqref{eq:trasi:domain}.
  \begin{theorem}\label{thm:medium}
For fixed $k_s, k_p \in \mathbb{R}_{+}$, we have the following asymptotic expansion for the elastic medium scattering problem \eqref{eq:scat:medium},
  \begin{align}\label{eq:rrr}
 u_{\omega,\Omega}^s(x) & = \frac{e^{ik_p|z|-ik_p\hat{z} \cdot y_0}}{4 \pi |z|}u_{p}^s(D,\hat{z},p;x-z) + \frac{e^{ik_s|z|-ik_s\hat{z} \cdot y_0}}{4 \pi |z|}u_{s}^s(D,\hat{z},p;x-z)+ \mathcal{O}(|z|^{-2}),
 \end{align}
 for any fixed $x \in \mathbb{R}^3$ as $|z| \rightarrow \infty$ uniformly for all $\hat{z} \in \mathbb{S}^2$. In \eqref{eq:rrr}, $u_{p}^s(D,\hat{z},p;x)$ and $u_{s}^s(D,\hat{z},p;x)$ are the scattered fields of the penetrable medium $D$ as in \eqref{eq:scat:medium:plane} corresponding to the incident plane
 elastic waves
  $u_{p}^i(x, \hat{z}, p)$ and $u_{s}^i(x, \hat{z},p)$ respectively.
 \end{theorem}

\begin{proof}
By \cite[Lemma 2]{PH0} and \cite[Lemma 5.7 ]{PHhab}, the radiating scattering elastic wave field has the following
integral representation,
\begin{equation}\label{eq:repre:medium}
u_{\omega,\Omega}^{s}=  -\omega^2 \int_{\Omega}\Gamma(x,y)n_{\Omega}(y)u_{\omega,\Omega}(y)dy.
\end{equation}
In the following, we set $-\omega^2 \int_{\Omega}\Gamma(x,y)n_{\Omega}(y)u_{\omega,\Omega}(y)dy:= \tilde{V}_{\omega,\Omega}u_{\omega,\Omega}$.
It is known that the operator $(I -\tilde{V}_{\omega,\Omega} )$ is continuously invertible in $C(\Omega)^3$ \cite{PHhab}. Thus we have
\begin{equation}\label{eq:eq:elec}
u_{\omega,\Omega}  = (I - \tilde{V}_{\omega,\Omega})^{-1}u_{\omega,p,ep}^{i}(\cdot, y_0),\quad x \in \Omega, \quad y_0 \notin \bar \Omega.
\end{equation}
Considering $u_{\omega,\Omega}^s(x+z)$, by \eqref{eq:repre:medium}, it can be written as
\[
u_{\omega,\Omega}^s(x+z)= -\omega^2 \int_{\Omega}\Gamma(x+z,y)n_{\Omega}(y)u_{\omega,\Omega}(y)dy.
\]
Setting $t=y-z$ and denoting $u_{\omega,D}(t+z)|_{t \in D}= u_{\omega,\Omega}(y)|_{y \in \Omega}$, by change of variables and noting that the Jacobian matrix of the change of variables is the identity matrix in $\mathbb{R}^3$, together with \eqref{eq:trasi:domain} and \eqref{eq:n:medium}, we have
\begin{equation*}
u_{\omega,\Omega}^s(x+z)  =  -\omega^2 \int_{D}\Gamma(x,t)n_{D}(t)u_{\omega,D}(t+z)dy.
\end{equation*}
It can be readily verified that
\begin{equation}\label{eq:e:omega:z}
u_{\omega,\Omega}^s(x+z)  =  (\tilde{V}_{\omega,D} u_{\omega,D}(t+z))(x).
\end{equation}
Next, we solve for $u_{\omega,D}(t+z)$. Rewriting \eqref{eq:eq:elec} as $(I - \tilde{V}_{\omega,\Omega})u_{\omega,\Omega}(x)  = u_{\omega,p,ep}^{i}(x,y_0)$, we have
\begin{equation}\label{eq:e:omega:medium}
u_{\omega,p,ep}^{i}(x,y_0)=u_{\omega,\Omega}(x) + \omega^2 \int_{\Omega}\Gamma(x,y)n_{\Omega}(y)u_{\omega, \Omega}(y)dy.
\end{equation}
By change of variables with $x = t+z$, \eqref{eq:e:omega:medium} becomes
\begin{equation}\label{eq:mega:2}
u_{\omega,p,ep}^{i}(t+z,y_0)-u_{\omega,D}(t+z)=\omega^2 \int_{\Omega}\Gamma(t+z,y)n_{\Omega}(y)u_{\omega, \Omega}(y)dy.
\end{equation}
By using change of variables with $\tilde y = y-z$, we see \eqref{eq:mega:2} can be written as
\begin{equation}\label{eq:volume:reso:add}
u_{\omega,p,ep}^{i}(t+z,y_0)-u_{\omega,D}(t+z)=\omega^2 \int_{D}\Gamma(t,\tilde y)n_{D}(\tilde y)u_{\omega,D}(\tilde y +z)d\tilde y.
\end{equation}
With \eqref{eq:volume:reso:add}, it can be shown
\begin{equation*}
u_{\omega,D}(t+z) = (I - \tilde{V}_{\omega,D})^{-1}u_{\omega,p,ep}^{i}(\cdot+z,y_0), \quad t \in D,
\end{equation*}
where $\tilde{V}_{\omega,D}$ is defined as follows,
\begin{equation}\label{eq:volume:p:D}
(\tilde{V}_{\omega,D} u)(x): = -\omega^2\int_{D} \Gamma(x,y) n_{D}(y) u(y)dy, \quad u \in C(D)^3.
\end{equation}
Substituting it into \eqref{eq:e:omega:z}, we have
\begin{equation*}
u_{\omega,\Omega}^s(x+z)  =  [\tilde{V}_{\omega,D} (I - \tilde{V}_{\omega,D})^{-1}u_{\omega,p,ep}^{i}(\cdot+z,y_0)](x).
\end{equation*}
Again by Lemma \ref{lem:incident}, we have
\begin{align}
u_{\omega,\Omega}^s(x+z)  =&  \frac{e^{ik_p|z|-ik_p\hat{z} \cdot y_0}}{4 \pi |z|}[\tilde{V}_{\omega,D} (I - \tilde{V}_{\omega,D})^{-1}u_{p}^{i}(x,\hat{z},p)(x) + \mathcal{O}(|z|^{-1})], \label{eq:aym:me:1} \\
& +\frac{e^{ik_s|z|-ik_s\hat{z} \cdot y_0}}{4 \pi |z|}[\tilde{V}_{\omega,D} (I - \tilde{V}_{\omega,D})^{-1}u_{s}^{i}(x,\hat{z},p)(x) + \mathcal{O}(|z|^{-1})]. \label{eq:aym:me:2}
\end{align}
 Now we introduce $u_{p}^{s}(D, \hat{z},p;x) := V_{\omega,D}(- u_{p}^{i}(t,\hat{z},p))(x)$ and $u_{s}^{s}(D, \hat{z},p;x) := V_{\omega,D}(-u_{s}^{i}(t,\hat{z},p))(x)$ with $V_{\omega,D}$ defined as follows,
    \[
    V_{\omega,D} : C(D)^3 \rightarrow C(D)^3, \quad V_{\omega,D} := \tilde{V}_{\omega,D} (I - \tilde{V}_{\omega,D})^{-1}.
    \]
    In fact, $u_{p}^{s}(D, \hat{z},p;x)$ and $u_{s}^{s}(D, \hat{z},p;x)$ are the scattering fields associated with the incident plane elastic wave $u_{p}^i(x, \hat{z},p)$ and $u_{s}^i(x, \hat{z},p)$, respectively.
Again by change of variables, we have
\begin{equation}\label{eq:medium:syste}
u_{\omega,\Omega}^s(x) = \frac{e^{ik_p|z|-ik_p\hat{z} \cdot y_0}}{4 \pi |z|}u_{p}^s(D,\hat{z},p;{x-z})+\frac{e^{ik_s|z|-ik_s\hat{z} \cdot y_0}}{4 \pi |z|}u_{s}^s(D,\hat{z},p;{x-z})+ \mathcal{O}(|z|^{-2}).
\end{equation}
\end{proof}


\section{Low Frequency Asymptotic Analysis}\label{sec:analysis:lowfre}
In this section, we discuss the low frequency asymptotic approximations of the nearly incompressible material for our algorithmic development.
There are some existing results on low frequency asymptotic analysis for elastic wave scattering \cite{DR}. However,  there are no direct and clear results that could cover the case of the nearly incompressible materials for our purpose, and thus we provide a complete study in what follows. In the sequel, for notational convenience and without loss of generality, we assume that the elastic point source \eqref{eq:point:source} is always located at the origin, i.e.,
$u_{\omega,p,ep}^{i}(x,0)=\Gamma(x,0)p$.

For the low frequency analysis, we first consider the limiting case with $\omega=0$.
We introduce $\Gamma^0(x)$ as in \cite{PH}, which is the fundamental solution of \eqref{eq:elastic:wave} with $\omega=0$, $\rho \equiv 1$
\begin{equation}\label{eq:gamma0}
\Gamma_{ij}^0(x): = \frac{\delta_{jk}}{4 \pi \mu |x|} - \frac{\lambda + \mu}{8\pi \mu(2\mu + \lambda)} \frac{\partial |x|}{\partial x_j \partial x_k},
\end{equation}
and define $\Pi^0(x,y)$ as follows,
\[
\Pi^0(x,y)^{T}d: = \mathcal{P}_y(\Gamma^0(x-y)d).
\]
Then the single and double layers associated with $\Gamma_{i,j}^0(x,y)$ can be defined similarly,
\begin{equation}\label{eq:gamma0:layer}
(S^{0}_{D}\phi)(x): = 2\int_{\partial D} \Gamma^{0}(x,y)\phi(y)ds(y), \quad (K^{0}_{D}\phi)(x): = 2\int_{\partial D} \Pi^{0}(x,y)\phi(y)ds(y).
\end{equation}
In fact, as $\omega \rightarrow 0$, the single layer $S^{0}_{D}$ and the double layer $K^{0}_{D}$ approximate $S_{\omega,D}$ and $K_{\omega,D}$, respectively, as shown in the lemma below.
\begin{lemma}\label{lem:layer:diff}
As the frequency $\omega \rightarrow 0$, we have the following estimates,
\begin{equation}\label{eq:estimates:layers}
||S_{\omega, D} - S_{D}^{0}||_{\mathcal{L}(C(\partial D)^3, C^{0, \alpha}(\partial D)^3)}  \sim \mathcal{O}(\omega), \quad ||K_{\omega, D} - K_{D}^{0}||_{\mathcal{L}(C(\partial D), C^{0, \alpha}(\partial D)^3)}  \sim \mathcal{O}(\omega).
\end{equation}
\end{lemma}
\begin{proof}
By \cite[Lemma 5.1]{PHhab}, as $\kappa\rightarrow 0$ with $\kappa = k_{p}$ or $\kappa = k_s$, we have
\begin{align*}
\frac{e^{i \kappa x}}{4\pi |x|}  &= \frac{cos(\kappa |x|)}{4 \pi |x|} + \frac{sin(\kappa |x|)}{4 \pi |x|} \\
 & = \frac{1}{4\pi |x|} - \frac{\kappa^2}{8\pi}|x| + \kappa^4|x|^3f_1(\kappa^2|x|^2) + i\kappa f_2(\kappa^2|x|^2),
\end{align*}
where $f_1$ and $f_2$ are entire functions. By direct calculation, as $\omega \rightarrow 0$, we see
\begin{equation}\label{gamma:gamma0:first}
\frac{e^{ik_sx}}{4\pi \mu |x|} - \frac{1}{4 \pi \mu |x|} = \frac{1}{\mu}(- \frac{k_s^2}{8\pi}|x| + k_s^4|x|^3f_1(k_s^2|x|^2) + ik_s f_2(k_s^2|x|^2),
\end{equation}
and by $(k_p^2-k_s^2)/\omega^2 = (\lambda+ \mu)/[\mu(2\mu + \lambda)]$, we also have
\begin{align}
 &\frac{\partial^2}{\partial x_j \partial x_k} \left[ \frac{1}{\omega^2}\frac{e^{ik_sx} -e^{ik_px}}{4 \pi |x|} + \frac{(\lambda+\mu)|x|}{8\pi \mu(2\mu + \lambda)}\right]\\
  = & \omega \frac{\partial^2}{\partial x_j \partial x_k}\left[ |x|^3g_{1}(|x|^2)+ig_{2}(|x|^2) +|x|^3h_{1}(|x|^2)+ih_{2}(|x|^2) \right], \label{gamma:gamma0:second}
\end{align}
where $g_1$, $g_2$, $h_1$ and $h_2$ are entire functions whose coefficients only depend on $\mu$, $\lambda$, $\omega^k$ with $k \in \mathbb{N}$.

Hence, for $\omega \ll 1$, and any constant $c_1$, there exist a constant $c_2$, such that for all $|x| \leq c_1$, by \eqref{gamma:gamma0:first} and the expressions of $\Gamma$ and $\Gamma^0$,  we have
\begin{equation}\label{eq:weak:singular:S}
|\Gamma_{jk}(x)-\Gamma_{jk}^0(x)| \leq c_2 \omega,  \quad j,k=1,2,3.
\end{equation}
By differentiating \eqref{gamma:gamma0:first}, and noting for odd integer $l$ the relation $\partial |x|^{l}/\partial x_m=lx_m|x|^{l-2}$, we have
\begin{equation}\label{eq:weak:singular:K}
|\frac{\partial}{\partial x_{l_1}}(\Gamma_{jk}(x)-\Gamma_{jk}^0(x))| \leq \frac{c_2\omega}{|x|} \Rightarrow |(\Pi(x)^{t}-{\Pi^0}^t(x))e_j| \leq \frac{c_2\omega}{|x|}, \quad \quad l_1, l_2, l_3=1,2,3,
\end{equation}
where $e_j$, $j=1,2,3$, denote the usual Cartesian orthonormal unit basis vector.

By \eqref{eq:weak:singular:S} and \eqref{eq:weak:singular:K} and Chapter 2.3 of \cite{CK0}, since the integral kernel of $S_{\omega, D} - S_{D}^{0}$ and $K_{\omega, D} - K_{D}^{0}$ are weakly singular, we see the mapping properties of them are at least as good as $S_{\omega, D}$ or $K_{\omega, D}$, which are bounded and linear mappings from $C(\partial D)^3$ to $C^{0,\alpha}(\partial D)^3$ . This leads to the desired estimates and completes the proof.
\end{proof}

It is necessary to introduce the far field patterns of any radiating elastic waves $u^{s}(x)$ for convenience \cite{AK, PH},
\[
u^{s}(x) = \frac{e^{ik_{s}|x|}}{|x|}u_{s,\infty}(\hat{x}) + \frac{e^{ik_{p}|x|}}{|x|}u_{p,\infty}(\hat{x})\hat{x} + \mathcal{O}(|x|^{-2}),
\]
for $|x| \rightarrow \infty$ where $\hat{x}\cdot u_{s,\infty}(\hat{x})=0$ for all $\hat{x} \in \mathbb{S}^2$ and $u_{p,\infty}(\hat{x})$ is complex valued scalar function.

With Lemma \ref{lem:layer:diff} and Theorem \ref{thm:pec}, we are now in the position to give the theorem on the low frequency asymptotic approximations of the nearly incompressible materials.
\begin{theorem}\label{thm:obstacle:asym}
Assuming $\Omega$ and $D$ be rigid obstacles with the translation \eqref{eq:trasi:domain} and $\lambda  \gg \mu$, $ \mu = \mO(1)$, $|z| \gg 1$, we
have the following asymptotic estimate of the scattered fields of \eqref{eq:scat:pec} as $|x| \rightarrow \infty$,  $\omega \rightarrow 0$,
\begin{align}
 u_{\omega,\Omega}^s(x)  =& \frac{e^{ik_p|z|}}{4 \pi |z|} \left[ \frac{e^{ik_s|x-z|}}{|x-z|} (\mO(\frac{1}{\lambda})(I-\hat{x}\hat{x}^T)a_0 + \frac{e^{ik_p|x-z|}}{|x-z|}\mO(\frac{1}{\lambda^2}) \hat{x}\hat{x}^{T}b_0 \right] + \mO(\frac{1}{|z|^2}) \label{eq:asym:obs:lambda}\\
 &+\frac{e^{ik_s|z|}}{4 \pi |z|} \left[ \frac{e^{ik_s|x-z|}}{|x-z|} (\mO(1)(I-\hat{x}\hat{x}^T)c_0  + \frac{e^{ik_p|x-z|}}{|x-z|}
 \mO(\frac{1}{\lambda}) \hat{x}\hat{x}^{T}d_0\right] + \mathcal{O}(\frac{1}{|x|^2}) + \mO(\omega). \notag
\end{align}
where $a_0$, $b_0$, $c_0$ and $d_0$ are all constant complex vectors that do not depend on $\omega$ and $\lambda$.
\end{theorem}
\begin{proof}
Here we assume that $\partial D$ is $C^2$-continuous such that we use the mapping properties of the layer potentials $K_{\omega, D}$ and $S_{\omega, D}$ in \cite{PH}.
We assume $u^s(D,d, p;x)$ has the following combined layer potential representation \cite{PH} with incident plane wave $u^i(x,d,p)$,
\[
u^s(D,d, p;x) = [(K_{\omega, D} + i S_{\omega, D})\phi](x).
\]
It is well known that $u^s(D,d, p;x)$ has the following asymptotic behavior as $|x|\rightarrow \infty$ \cite{AK},
\begin{equation}\label{eq:scatter:repre:rigi}
 u^s(D,d, p;x)= \frac{e^{ik_s|x|}}{|x|}(\mathcal{F}_{s,\infty}\phi)(\hat{x})+ \frac{e^{ik_p|x|}}{|x|} (\mathcal{F}_{p,\infty}\phi)(\hat{x}) + \mathcal{O}(\frac{1}{|x|^2}),
\end{equation}
where $(\mathcal{F}_{s,\infty}\phi)(\hat{x})$ is the shear far
field pattern that is tangential to $\mathbb{S}^2$, and $(\mathcal{F}_{p,\infty}\phi)(\hat{x})$ is the pressure far field pattern that is normal to $\mathbb{S}^2$.
Beside, they have the the following integral representations \cite{AK},
\begin{align}
(\mathcal{F}_{s,\infty}\phi)(\hat{x})&= \frac{1}{2\pi \mu}\int_{\partial D}\left\{ [I -\hat{x}\hat{x}^{T}]e^{-ik_s\hat{x}\cdot y} + i [\mathcal{P}_{y}(I-\hat{x}\hat{x}^{T})e^{-ik_s\hat{x}\cdot y}]^{T}\right\}\phi(y)ds(y),\label{eq:far:s:asy}\\
(\mathcal{F}_{p,\infty}\phi)(\hat{x})& = \frac{1}{2\pi (2\mu + \lambda)}\int_{\partial D}\left\{ \hat{x}\hat{x}^{T}e^{-ik_p\hat{x}\cdot y} + i [\mathcal{P}_{y}\hat{x}\hat{x}^{T}e^{-ik_p\hat{x}\cdot y}]^{T}\right\}\phi(y)ds(y).  \label{eq:far:p:asy}
\end{align}
By \eqref{eq:plane}, as $\omega \rightarrow 0$, we have $\lim_{\omega \rightarrow 0}u^i(x,d,p) = u^0 = u_{s}^{0} + u_{p}^0$, where
\begin{equation}\label{eq:lowlimit}
u_{s}^{0}=  \lim_{\omega \rightarrow 0} u_{s}^i =  \frac{1}{\mu} (d\times p)\times d,\ \quad u_{p}^0 =  \lim_{\omega \rightarrow 0} u_{p}^i = \frac{1}{2\mu + \lambda}(d\cdot p) d.
\end{equation}

Next we introduce $\phi^{\omega}(x) : = \phi_{p}^{\omega}(x) + \phi_{s}^{\omega}(x)$,
\[
\phi_{p}^{\omega}(x) := (I + K_{\omega, D} + i S_{\omega, D})^{-1}(-u_{p}^{i}(\cdot, d,p)), \ \ \phi_{s}^{\omega}(x) := (I + K_{\omega, D} + i S_{\omega, D})^{-1}(-u_{s}^{i}(\cdot, d,p)).
\]
By Lemma \ref{lem:layer:diff}, it can be shown that as $\omega \rightarrow 0$,
\begin{align}
\phi^{\omega}(x) &= (I + K_{\omega, D} + i S_{\omega, D})^{-1}(-u^{i})  =   [I + K_{D}^{0} + i S_{D}^{0} + \mathcal{O}(\omega)]^{-1}(-u^0 + \mathcal{O}(\omega)) \notag \\
        &= -[I + K_{D}^{0} + i S_{D}^{0}]^{-1}u^0 + \mathcal{O}(\omega). \label{eq:}
\end{align}
Thus we arrive at
\[
\phi^{0}(x): = \lim_{\omega \rightarrow 0}\phi(x) =  [I + K_{D}^{0} + i S_{D}^{0}]^{-1}(-u^0)= \phi_{s}^0+ \phi_{p}^0,
\]
where
\begin{equation}\label{eq:phi:u:rela}
\phi_{s}^0: = [I + K_{D}^{0} + i S_{D}^{0}]^{-1}(-u_{s}^0),\quad \quad \phi_{p}^0: = [I + K_{D}^{0} + i S_{D}^{0}]^{-1}(- u_{p}^0).
\end{equation}
It then follows that
\begin{align*}
u_{p}^s(D,\hat{z},p;x) &= [K_{\omega,D} + iS_{\omega,D}]  (I + K_{\omega,D} + i S_{\omega,D})^{-1}(- u_{s}^{i}(t,\hat{z},p)) \\
& = [K_{\omega,D} + iS_{\omega,D}]\phi_{s}^0 + \mO(\omega) \\
&= \frac{e^{ik_s|x|}}{|x|} (\mathcal{F}_{s,\infty}\phi_{s}^0)(\hat{x}) + \frac{e^{ik_p|x|}}{|x|} (\mathcal{F}_{p,\infty}\phi_{s}^0)(\hat{x})+ \mathcal{O}(\frac{1}{|x|^2}) + \mO(\omega).
\end{align*}
Similarly, we have
\begin{equation}
u_{p}^s(D,\hat{z},p;x) =  \frac{e^{ik_s|x|}}{|x|} (\mathcal{F}_{s,\infty}\phi_{p}^0)(\hat{x}) + \frac{e^{ik_p|x|}}{|x|} (\mathcal{F}_{p,\infty}\phi_{p}^0)(\hat{x})+ \mathcal{O}(\frac{1}{|x|^2}) + \mO(\omega).
\end{equation}
By \eqref{eq:asym:rigid:twp} and Theorem \ref{thm:pec}, it is straightforward to show that
\begin{align*}
 u_{\omega,\Omega}^s(x+z)  =& \frac{e^{ik_p|z|}}{4 \pi |z|} \left[ \frac{e^{ik_s|x|}}{|x|} (\mathcal{F}_{s,\infty}\phi_{p}^0)(\hat{x}) + \frac{e^{ik_p|x|}}{|x|} (\mathcal{F}_{p,\infty}\phi_{p}^0)(\hat{x})\right] + \mathcal{O}(\frac{1}{|x|^2}) \\
 &+\frac{e^{ik_s|z|}}{4 \pi |z|} \left[ \frac{e^{ik_s|x|}}{|x|} (\mathcal{F}_{s,\infty}\phi_{s}^0)(\hat{x}) + \frac{e^{ik_p|x|}}{|x|} (\mathcal{F}_{p,\infty}\phi_{s}^0)(\hat{x})\right] + \mO(\frac{1}{|z|^2}) + \mO(\omega).
\end{align*}
Furthermore, we can show that
\begin{align}
 u_{\omega,\Omega}^s(x)  =& \frac{e^{ik_p|z|}}{4 \pi |z|} \left[ \frac{e^{ik_s|x-z|}}{|x-z|} (\mathcal{F}_{s,\infty}\phi_{p}^0)(\widehat{x-z}) + \frac{e^{ik_p|x-z|}}{|x-z|} (\mathcal{F}_{p,\infty}\phi_{p}^0)(\widehat{x-z})\right] + \mO(\frac{1}{|z|^2}) \label{eq:rigid:fullexpan} \\
 &+\frac{e^{ik_s|z|}}{4 \pi |z|} \left[ \frac{e^{ik_s|x-z|}}{|x-z|} (\mathcal{F}_{s,\infty}\phi_{s}^0)(\widehat{x-z}) + \frac{e^{ik_p|x-z|}}{|x-z|} (\mathcal{F}_{p,\infty}\phi_{s}^0)(\widehat{x-z})\right] + \mathcal{O}(\frac{1}{|x|^2}) + \mO(\omega). \notag
\end{align}
By combining \eqref{eq:far:s:asy}, \eqref{eq:far:p:asy}, \eqref{eq:lowlimit} and \eqref{eq:phi:u:rela}, as $|x| \rightarrow \infty$ and $\omega \rightarrow 0$, we have
\begin{align*}
&(\mathcal{F}_{s,\infty}\phi_{p}^0)(\widehat{x-z}) =(I -\hat{x}\hat{x}^T)\mO(\frac{1}{\lambda})a_0 , \quad (\mathcal{F}_{p,\infty}\phi_{p}^0)(\widehat{x-z}) = \hat{x}\hat{x}^T\mO(\frac{1}{\lambda^2})b_0, \\
&(\mathcal{F}_{s,\infty}\phi_{s}^0)(\widehat{x-z}) =(I- \hat{x}\hat{x}^{T})\mO(1)c_0, \quad  (\mathcal{F}_{p,\infty}\phi_{s}^0)(\widehat{x-z}) =\hat{x}\hat{x}^T \mO(\frac{1}{\lambda})d_0,
\end{align*}
where $a_0$, $b_0$, $c_0$ and $d_0$ are constant vectors that do not depend on $\lambda$ and $\omega$.

The proof is complete.

\end{proof}

For the scattering of the rigid obstacle under a regular frequency, we have the following corollary.

\begin{corollary}\label{cor:near:shape:rigid:dete}
Assuming $\Omega$ and $D$ be rigid obstacles with the translation \eqref{eq:trasi:domain} and $\lambda  \gg \mu$, $ \mu = \mO(1)$, $|z| \gg 1$, we
have the following asymptotic estimate of the scattered fields of \eqref{eq:scat:pec} as $|x| \rightarrow \infty$, with the incident plane wave $u^{i}(x,d,p)=u_{p}^i + u_{s}^i$ as in \eqref{eq:plane}
\begin{subequations}  \label{eq:normal:fre:repre}
\begin{align}
 u_{\omega,\Omega}^s(x) =& \frac{e^{ik_s|z|}}{4 \pi |z|} \left[ \frac{e^{ik_s|x-z|}}{|x-z|} (\mathcal{F}_{s,\infty}\phi_{s}^\omega)(\widehat{x-z})   + \frac{e^{ik_p|x-z|}}{|x-z|}
(\mathcal{F}_{p,\infty}\phi_{s}^\omega)(\widehat{x-z}) \right] + \mathcal{O}(\frac{1}{|x|^2}) \\
   &+\frac{e^{ik_p|z|}}{4 \pi |z|} \left[ \frac{e^{ik_s|x-z|}}{|x-z|}  (\mathcal{F}_{s,\infty}\phi_{p}^\omega)(\widehat{x-z}) + \frac{e^{ik_p|x-z|}}{|x-z|} (\mathcal{F}_{p,\infty}\phi_{p}^\omega)(\widehat{x-z})\right] + \mO(\frac{1}{|z|^2}).
\end{align}
\end{subequations}
Furthermore,  we have
\begin{align*}
&|(\mathcal{F}_{s,\infty}\phi_{p}^\omega)(\widehat{x-z})| = \mO(\frac{1}{\lambda}), \quad |(\mathcal{F}_{p,\infty}\phi_{p}^\omega)(\widehat{x-z})| = \mO(\frac{1}{\lambda^2}), \\
&|(\mathcal{F}_{s,\infty}\phi_{s}^\omega)(\widehat{x-z})| = \mO(1), \quad  |(\mathcal{F}_{p,\infty}\phi_{s}^\omega)(\widehat{x-z})| = \mO(\frac{1}{\lambda}).
\end{align*}
\end{corollary}
\begin{proof}
By \eqref{eq:asym:plane:incident} and \eqref{eq:asym:rigid:twp}, and the boundedness of the linear operator $(I + K_{\omega,D} + i S_{\omega,D})^{-1}$ in $C(\partial D)^3$ \cite{KK, PH}, we have
\begin{align}
&|  (I + K_{\omega,D} + i S_{\omega,D})^{-1}(- u_{p}^{i}(t,\hat{z},p))| \sim \mO(\frac{1}{\lambda}), \\
&|  (I + K_{\omega,D} + i S_{\omega,D})^{-1}(- u_{s}^{i}(t,\hat{z},p))| \sim \mO(\frac{1}{\mu}). \label{eq:asym:normal:fre}
\end{align}
By the integral representation $[K_{\omega,D} + iS_{\omega,D}]$ in \eqref{eq:scatter:repre:rigi}, and the representations of
$\mathcal{F}_{s,\infty}$ in \eqref{eq:far:s:asy} and $\mathcal{F}_{p,\infty}$ in \eqref{eq:far:p:asy}, we can obtain \eqref{eq:normal:fre:repre}. Furthermore with \eqref{eq:asym:plane:incident}, \eqref{eq:far:s:asy}, and \eqref{eq:far:p:asy} and \eqref{eq:normal:fre:repre}, we can obtain their asymptotic magnitudes as $\lambda \rightarrow \infty$.
\end{proof}
For the scattering of the penetrable medium case, we have a similar theorem for the low frequency scattering.
\begin{theorem}\label{thm:medium:asym}
Let $\Omega$ and $D$ be medium scatterers as described earlier with the translation \eqref{eq:trasi:domain}. Suppose $\lambda  \gg \mu$ and $ \mu = \mO(1)$. Then as $|z| \gg 1$ and $\omega \ll 1$ and denoting $\mO(x,z) = \mO(|z|^{-1}) + \mO(|x|^{-2})$, we
have the following asymptotic estimate of the scattered fields of \eqref{eq:scat:medium} with $|x| \rightarrow \infty$,
\begin{align}
 &u_{\omega,\Omega}^s(x)  = -\omega^2\frac{e^{ik_p|z|}}{4 \pi |z|} \left[ \frac{e^{ik_s|x-z|}}{|x-z|} (\mO(\frac{1}{\lambda})(I-\hat{x}\hat{x}^T)a^m_0 + \frac{e^{ik_p|x-z|}}{|x-z|}\mO(\frac{1}{\lambda^2}) \hat{x}\hat{x}^{T}b^m_0  + \mO(x,z)  \right] \notag \\
 &-\omega^2\frac{e^{ik_s|z|}}{4 \pi |z|} \left[ \frac{e^{ik_s|x-z|}}{|x-z|} (\mO(1)(I-\hat{x}\hat{x}^T)c^m_0  + \frac{e^{ik_p|x-z|}}{|x-z|}
 \mO(\frac{1}{\lambda}) \hat{x}\hat{x}^{T}d^m_0 +  \mO(x,z) \right] + \smallO(\omega^2). \notag
\end{align}
where $a^m_0$, $b^m_0$, $c^m_0$ and $d^m_0$ are all constant vectors that do not depend on $\omega$ and $\lambda$.
\end{theorem}
\begin{proof}
The proof is similar to those of Lemma \ref{lem:layer:diff} and Theorem \ref{thm:obstacle:asym}. We introduce the following volume potential
\[
 (\tilde{V}_{D}^{0}\varphi)(x)  :=   - \int_{D}\Gamma^{0}(x,y)n_{D}(y)\varphi(y)dy.
\]
By \eqref{eq:weak:singular:S} and \cite[Chap. 5]{PHhab}, we see as $\omega \rightarrow 0$
\begin{equation}
 \|\tilde{V}_{\omega,D}-\omega^2\tilde{V}_{D}^0\|_{C(D)\rightarrow C^{1, \gamma}(D)} \sim \|n_{D}\|_{C(D)}\mathcal{O}(\omega^3).
\end{equation}
For the scattering by plane incident wave of the medium $D$ \eqref{eq:scat:medium}, as $\omega \rightarrow 0$, we have
 \begin{equation}\label{eq:asym:medium:total:1}
 u(D,d, p;x) = (I- \tilde{V}_{\omega,D})^{-1}u_{\omega, p,d}^{i} =  (I - \omega^2 \tilde{V}_{D}^{0} +\mathcal{O}(\omega) )^{-1}(u^0 + \mathcal{O}(\omega)) = u^0 + \mathcal{O}(\omega^2).
 \end{equation}
In addition, the fundamental solution matrix $\Gamma(x,y)$ has the following asymptotic behavior \cite{Kup},
 \begin{equation}\label{eq:asym:elastic:funda}
 \Gamma(x,y) = \frac{e^{ik_p|x|}}{|x|}\frac{\hat{x}\hat{x}^{T}e^{-ik_p \hat{x}\cdot y}}{4\pi(\lambda+2\mu)}
 + \frac{e^{ik_s|x|}}{|x|}\frac{1}{4\pi \mu}(I-\hat{x}\hat{x}^{T})e^{-ik_s\hat{x}\cdot y} + \mathcal{O}(\frac{1}{|x|^2}), \ |x| \rightarrow \infty.
 \end{equation}
By the integral representation \eqref{eq:volume:p:D} of the scattered field of \eqref{eq:scat:medium}, $ u^s(D,d, p;x) = [\tilde{V}_{\omega,D}u(D,d,p;\cdot)](x)$, together with \eqref{eq:asym:medium:total:1} and \eqref{eq:asym:elastic:funda},
we have
\begin{equation*}
 u^{\infty}(D,d, p;\hat{x}) = -\omega^2 \left[ \frac{e^{ik_p|x|}}{|x|} (\mathcal{F}^{m}_{p,\infty} u(D,d, p;\cdot))(\hat{x}) +   \frac{e^{ik_s|x|}}{|x|}(\mathcal{F}^{m}_{s,\infty} u(D,d, p;\cdot))(\hat{x}) \right],
\end{equation*}
where the far field mapping operators $\mathcal{F}^{m}_{s,\infty}$ and $\mathcal{F}^{m}_{p,\infty}$ are defined by
\begin{align}
(\mathcal{F}^{m}_{s,\infty}\phi)(\hat{x})&:= \frac{1}{4\pi \mu}\int_{D} [I -\hat{x}\hat{x}^{T}]e^{-ik_s\hat{x}\cdot y} n_{D}(y) \phi(y) dy,\label{eq:far:s:asy:medium}\\
(\mathcal{F}^{m}_{p,\infty}\phi)(\hat{x})& := \frac{1}{4\pi (2\mu + \lambda)}\int_{D} \hat{x}\hat{x}^{T}e^{-ik_p\hat{x}\cdot y}n_{D}(y) \phi(y) dy.  \label{eq:far:p:asy:medium}
\end{align}
Recalling \eqref{eq:medium:syste}, together with \eqref{eq:lowlimit} and \eqref{thm:medium:asym}, similar to the rigid obstacle case \eqref{eq:rigid:fullexpan}, we see
\begin{align}
 &u_{\omega,\Omega}^s(x)  = -\omega^2\frac{e^{ik_p|z|}}{4 \pi |z|} \left[ \frac{e^{ik_s|x-z|}}{|x-z|} (\mathcal{F}^m_{s,\infty}u_{p}^0)(\widehat{x-z}) + \frac{e^{ik_p|x-z|}}{|x-z|} (\mathcal{F}^m_{p,\infty}u_{p}^0)(\widehat{x-z}) + \mO(x,z) \right] \notag \\
 &-\omega^2\frac{e^{ik_s|z|}}{4 \pi |z|} \left[ \frac{e^{ik_s|x-z|}}{|x-z|} (\mathcal{F}^m_{s,\infty}u_{s}^0)(\widehat{x-z}) + \frac{e^{ik_p|x-z|}}{|x-z|} (\mathcal{F}^m_{p,\infty}u_{s}^0)(\widehat{x-z}) +  \mO(x,z)\right] + \smallO(\omega^2).  \label{eq:medium:fullexpan}
\end{align}
Combining \eqref{eq:lowlimit}, \eqref{eq:far:s:asy:medium} and \eqref{eq:far:p:asy:medium}, as $|x| \rightarrow \infty$ and $\omega \rightarrow 0$, we have
\begin{align}
&(\mathcal{F}^m_{s,\infty}u_{p}^0)(\widehat{x-z}) = (I-\hat{x}\hat{x}^{T})\mO(\frac{1}{\lambda})a^m_0 , \quad (\mathcal{F}^m_{p,\infty}u_{p}^0)(\widehat{x-z}) = \hat{x}\hat{x}^{T}\mO(\frac{1}{\lambda^2})b^m_0, \\
&(\mathcal{F}^m_{s,\infty}u_{s}^0)(\widehat{x-z}) = (I-\hat{x}\hat{x}^{T})\mO(1)c^m_0, \quad  (\mathcal{F}^m_{p,\infty}u_{s}^0)(\widehat{x-z}) = \hat{x}\hat{x}^{T} \mO(\frac{1}{\lambda})d^m_0,
\end{align}
where $a^m_0$, $b^m_0$, $c^m_0$ and $d^m_0$ are all constant complex vectors that do not depend on $\omega$ and $\lambda$, which leads to this theorem.
\end{proof}
For the medium scattering under a regular frequency, we have the following result. We denote
\[
 v_{p}^{\omega}(x) = (I- \tilde{V}_{\omega,D})^{-1}u_{p}^{i}, \ \ v_{s}^{\omega}(x) = (I- \tilde{V}_{\omega,D})^{-1}u_{s}^{i}.
\]
\begin{corollary}\label{cor:asym:medium:shape:dete}
 Let $\Omega$ and $D$ be medium with the translation \eqref{eq:trasi:domain}. Supposing $\lambda  \gg \mu$, $ \mu = \mO(1)$, $|z| \gg 1$, then we
have the following asymptotic estimate of the scattered fields of \eqref{eq:scat:medium} as $|x| \rightarrow \infty$, with the incident plane wave $u^{i}(x,d,p)=u_{p}^i + u_{s}^i$
\begin{align}
 u_{\omega,\Omega}^s(x) =& -\omega^2\frac{e^{ik_s|z|}}{4 \pi |z|} \left[ \frac{e^{ik_s|x-z|}}{|x-z|} (\mathcal{F}^m_{s,\infty}v_{s}^\omega)(\widehat{x-z})   + \frac{e^{ik_p|x-z|}}{|x-z|}
(\mathcal{F}^m_{p,\infty}v_{s}^\omega)(\widehat{x-z}) \right] + \mathcal{O}(\frac{1}{|x|^2}) \\
   &-\omega^2\frac{e^{ik_p|z|}}{4 \pi |z|} \left[ \frac{e^{ik_s|x-z|}}{|x-z|}  (\mathcal{F}^m_{s,\infty}v_{p}^\omega)(\widehat{x-z}) + \frac{e^{ik_p|x-z|}}{|x-z|} (\mathcal{F}^m_{p,\infty}v_{p}^\omega)(\widehat{x-z})\right] + \mO(\frac{1}{|z|^2}). \notag
\end{align}
Additionally, we have
\begin{align*}
&|(\mathcal{F}^m_{s,\infty}v_{p}^\omega)(\widehat{x-z})| = \mO(\frac{1}{\lambda}), \quad |(\mathcal{F}^m_{p,\infty}v_{p}^\omega)(\widehat{x-z})| = \mO(\frac{1}{\lambda^2}), \\
&|(\mathcal{F}^m_{s,\infty}v_{s}^\omega)(\widehat{x-z})| = \mO(1), \quad  |(\mathcal{F}^m_{p,\infty}v_{s}^\omega)(\widehat{x-z})| = \mO(\frac{1}{\lambda}).
\end{align*}
\end{corollary}

The proof of Corollary \ref{cor:asym:medium:shape:dete} is completely similar to that of Theorem \ref{thm:medium:asym} and we skip it here.

\section{A Two-stage Reconstruction Algorithm}\label{sec:two-stage}
For the elastic wave scattering, Theorems \ref{thm:pec} and \ref{thm:medium} indicate that either with the incident plane pressure wave $u_{p}^i$ or with the incident plane shear wave $u_{s}^i$, the scattered field consists of both the shear and pressure waves which are difficult to separate one from the other. However, one can benefit from the orthogonality of the shear and pressure waves with only a single incident plane wave, i.e., with only $u_{p}^i$ or $u_{s}^i$. Similar ideas could be found in \cite{HLLS} and \cite{Sini}, where the direction and phase information of the far fields can be separated for reconstructions.
 For the incident point waves, by Lemma \ref{lem:incident}, two kinds of incident plane waves are mixed with complex phases even in the asymptotic expansion of point source waves. Fortunately, with the help of the low frequency analysis developed in Section \ref{sec:analysis:lowfre}, we can benefit from the dominating part of the point source wave for the nearly compressible materials.
This crucial observation, together with the results derived in Section \ref{sec:mathground}, enables us to propose a two-stage algorithm for finding both the locations and the shapes of the unknown elastic scatterers from an admissible dictionary. At the first stage, a low frequency is used for locating the scatterers, and at the second stage, a regular frequency is used for determining the shapes of the scatterers from the dictionary.

\subsection{Locating the scatterers}\label{subsect:locating}

We let $\Lambda$ denote the bounded measurement surface in $\mathbb{R}^3$, and $\tilde{z}$ be an arbitrary sampling point contained in a bounded sampling region $S \subseteq \mathbb{R}^3$. We also introduce $u_{\omega}^{s}(D,z;x):=u_{\omega,\Omega}^{s}(x)$ and $u_{\omega}^{\infty}(D,z;\hat{x}):=u_{\omega,\Omega}^{\infty}(\hat{x})$ as
the corresponding scattered field and far field of $\Omega$ due to the incident point source $u_{\omega, p, ep}^i(x,0)$ as in \eqref{eq:point:source}. These are the measurement data for our reconstruction schemes.

With the transition relations built in Theorems \ref{thm:pec} and \ref{thm:medium}, and the asymptotic estimates given in Theorems \ref{thm:obstacle:asym} and \ref{thm:medium:asym}, we propose the following two-stage algorithm for locating the position and determining the shape of the scatterer.

There are two separate cases in our subsequent discussion, depending on the noise level $\epsilon$ in the measured data. From the low frequency analysis in Theorems \ref{thm:obstacle:asym} and \ref{thm:medium:asym}, it can be seen that if the order of noise level is  $\smallO (\frac{1}{\lambda})$, then we can extract the $\mathcal{O}(\frac{1}{\lambda})$ order pressure wave alone for the reconstruction. If the noise level is of order $\mathcal{O}(\frac{1}{\lambda})$, the $\mathcal{O}(\frac{1}{\lambda})$ order pressure wave would be polluted. We then extract the dominating $\mathcal{O}(1)$ order shear wave instead for the reconstruction.

\medskip

Case I:~The noise level $\epsilon \sim  \smallO (\frac{1}{\lambda})$.
In this case, we only use the pressure wave field for the reconstruction and introduce the following imaging functional,
\begin{equation}\label{eq:indicator:p}
I_{p}(D,z, \tilde{z}): = \frac{|\langle P_{\hat{x}}[u_{\omega}^s(D,z;\cdot)], \mathring{u}_{p}^s(\tilde{z};\cdot)\rangle_{L^{2}(\Lambda)} |}{\|P_{\hat{x}}[u_{\omega}^s(D,z;\cdot)]\|_{L^2(\Lambda)} \|\mathring{u}_{p}^s(\tilde{z};\cdot)\|_{L^2(\Lambda)} }, \quad \tilde{z} \in S,
\end{equation}
where $u_{\omega,\Omega}^s(D,z;\cdot) $ is the scattered field produced by the incident point source $u_{\omega, p, ep}^i(x,0)$ as in \eqref{eq:point:source}. $P_{\hat{x}}=\hat{x}\hat{x}^{T}$ is the projection along the direction $\hat{x}$, and the test function $\mathring{u}_{p}^s(\tilde{z};x)$ is defined as
\begin{equation}
\mathring{u}_{p}^s(\tilde{z};x): = \frac{e^{ik_s|\tilde{z}|}}{4\pi |\tilde{z}|} \frac{e^{ik_p|x-\tilde{z}|}}{|x-\tilde{z}|}\hat{x}.
\end{equation}
If the measurement data are phaseless, then the imaging functional is modified as follows, and see \cite{Klib1, Klib2} for more reconstruction schemes and algorithms with phaseless data.

\begin{equation}\label{eq:indicator:p:phaseless}
I_{|p|}(D,z, \tilde{z}): = \frac{|\langle |P_{\hat{x}}[u_{\omega}^s(D,z;\cdot)]|, |\mathring{u}_{p}^s(\tilde{z};\cdot)|\rangle_{L^{2}(\Lambda)} |}{\|P_{\hat{x}}[u_{\omega}^s(D,z;\cdot)]\|_{L^2(\Lambda)} \|\mathring{u}_{p}^s(\tilde{z};\cdot)\|_{L^2(\Lambda)} }, \quad \tilde{z} \in S.
\end{equation}
\medskip

Case II: the noise level $ \epsilon \sim \mO(\frac{1}{\lambda})$.
In this case, we will use the shear far field for the reconstruction. According to Theorems \ref{thm:obstacle:asym} and \ref{thm:medium:asym}, although the pressure and shear parts are mixed together in the near scattered field, the shear far field is the dominant part with magnitude $\mO(\mu)$.
We introduce
\[
\mathring{u}_{k_s, m, H_1}^{\infty}(\tilde{z};x) := \frac{e^{ik_s|\tilde{z}|}}{4\pi |\tilde{z}|} e^{-ik_s \hat{x}\cdot \tilde{z}}U_{1}^{m}(\hat{x}),\quad \mathring{u}_{k_s, m, H_2}^{\infty}(\tilde{z};x) := \frac{e^{ik_s|\tilde{z}|}}{4\pi |\tilde{z}|} e^{-ik_s \hat{x}\cdot \tilde{z}}V_{1}^{m}(\hat{x}), \quad m=-1,0,1,
\]
where $U_{1}^m$ and $V_{1}^m$ are the vectorial spherical harmonics (cf.~\cite{CK}),
\[
U_{1}^m: = \frac{1}{2} \text{Grad}Y_{1}^m(\hat{x}), \quad V_{1}^m: = \frac{1}{2} \hat{x} \times \text{Grad}Y_{1}^m(\hat{x}), \quad m=-1,0,1.
\]
We propose the following imaging functional,
\begin{equation}\label{eq:indicator:s}
I_{s}(D,z, \tilde{z}): = \frac{\sqrt{\sum_{j=1}^2\sum_{m=-1}^1|\langle (I-P_{\hat{x}})u_{\omega}^{\infty}(D,z;\cdot), \mathring{u}_{k_s,m, H_j}^{\infty}(\tilde{z};\cdot)\rangle_{T^{2}(\mathbb{S}^2)}|^2}}{\|(I-P_{\hat{x}})u_{\omega}^{\infty}(D,z;\cdot)\|_{T^2(\mathbb{S}^2)} 1/{(4\pi |\tilde{z}|)}}, \quad \tilde{z} \in S,
\end{equation}
where $T^2(\mathbb{S}^2)$ is the tangential vector space on the unit sphere $\mathbb{S}^2$, and $(I-P_{\hat{x}})$ is the projection to the shear part,
\[
(I-P_{\hat{x}})u_{\omega}^{\infty}(D,z;\cdot) =u_{\omega}^{\infty}(D,z;\cdot) - \hat{x} \hat{x}^{T}u_{\omega}^{\infty}(D,z;\cdot).
\]

Next we show the indicating behaviours of the imaging functionals introduced in \eqref{eq:indicator:p}, \eqref{eq:indicator:p:phaseless} and \eqref{eq:indicator:s}, for locating the position $z$ of the scatterer $\Omega$.

\begin{theorem}
If $D$ is a rigid scatterer as in Theorem~\ref{thm:obstacle:asym} (or a penetrable medium given in Theorem \ref{thm:medium:asym}) and assume $|a_0|$, $|b_0|$, $|c_0|$, $|d_0|$ (or $|a^m_0|$, $|b^m_0|$, $|c^m_0|$, $|d^m_0|$ for medium case) are positive for all $D \in \mathfrak{D}$. Then we have the following asymptotic expansion
\begin{equation}\label{eq:indica:relation:p}
\lim_{\lambda \rightarrow \infty}\lim_{\omega \rightarrow 0} I_{p}(D,z, \tilde{z}) = \mathring{I}_{p}(z;\tilde{z})[1+ \mathcal{O}(|z|^{-1})], \quad |z|\rightarrow \infty, \quad |x| \rightarrow \infty,
\end{equation}
uniformly for all $D \in \mathfrak{D}$, $z \in \Lambda$ and $\tilde{z} \in \Lambda$, where for the indicator functional \eqref{eq:indicator:p}, one has
\begin{equation}
 \mathring{I}_{p}(z;\tilde{z})  = \frac{|\langle \mathring{u}_{p}^s(z;\cdot), \mathring{u}_{p}^s(\tilde{z};\cdot)  \rangle_{L^2(\Lambda)} |}{\|\mathring{u}_{p}^s({z};\cdot)\|_{L^2(\Lambda)} \|\mathring{u}_{p}^s(\tilde{z};\cdot)\|_{L^2(\Lambda)} }, \quad \tilde{z} \in S,
\end{equation}
and for the phaseless indicator functional \eqref{eq:indicator:p:phaseless}, denoting $ \mathring{I}_{|p|}$ as the similar limit of $I_{|p|}$ as in \eqref{eq:indica:relation:p}, one has
\begin{equation}
 \mathring{I}_{|p|}(z;\tilde{z})  = \frac{|\langle |\mathring{u}_{p}^s(z;\cdot)|, |\mathring{u}_{p}^s(\tilde{z};\cdot) | \rangle_{L^2(\Lambda)} |}{\|\mathring{u}_{p}^s({z};\cdot) \|_{L^2(\Lambda)}\|\mathring{u}_{p}^s(\tilde{z};\cdot)\|_{L^2(\Lambda)} }, \quad \tilde{z} \in S.
\end{equation}
Similarly for the indicator $I_{s}(D,z, \tilde{z})$, we have
\begin{equation}\label{eq:indica:relation:s}
\lim_{\lambda \rightarrow \infty}\lim_{\omega \rightarrow 0} I_{s}(D,z, \tilde{z}) = \mathring{I}_{s}(z;\tilde{z})[1+ \mathcal{O}(|z|^{-1})], \quad |z|\rightarrow \infty, \quad |x| \rightarrow \infty,
\end{equation}
uniformly for all $D \in \mathfrak{D}$, $\hat{z}\in \mathbb{S}^2$ and $\tilde{z} \in S$, where

\begin{equation}\label{eq:indicator:asym}
\mathring{I}_{s}(z, \tilde{z}): = \frac{\sqrt{\sum_{j=1}^2\sum_{m=-1}^1|\langle \tilde{u}_{k_s}^{\infty}(D,z;\cdot), \mathring{u}_{k_s,m, H_j}^{\infty}(\tilde{z};\cdot)\rangle_{T^{2}(\mathbb{S}^2)}|^2}}{\|\tilde{u}_{k_s}^{\infty}(D,z;\cdot)\|_{T^2(\mathbb{S}^2)} 1/{(4\pi |\tilde{z}|)}}, \quad \tilde{z} \in S,
\end{equation}
with
\[
\tilde{u}_{k_s}^{\infty}(D,z;\cdot) =  \frac{e^{ik_s|z|}}{4 \pi |z|} e^{-ik_s\hat{x}\cdot z} \begin{cases} (\mathcal{F}_{s,\infty}\phi_{s}^0)(\widehat{x-z}), \quad D \ \text{rigid obstacle}, \ \phi_{s}^0 \  \text{as in} \ \eqref{eq:phi:u:rela}, \ \\
(\mathcal{F}^m_{s,\infty}u_{s}^0)(\widehat{x-z}), \quad D \  \text{medium}, \ u_{s}^0 \  \text{as in}\  \eqref{eq:medium:fullexpan}.
\end{cases}
\]

\end{theorem}
\begin{proof}
We only prove the case that $D$ is a rigid obstacle with the indicator functional given by $I_{p}(D,z, \tilde{z})$. The other cases can be proved in a similar manner. By Theorems \ref{thm:pec} and \ref{thm:obstacle:asym}, there exists a constant vector $d_0$, such that
as $|z|\rightarrow \infty$,  we have by \eqref{eq:asym:obs:lambda},
\begin{align}
\lim_{\omega \rightarrow 0} P_{\hat{x}}[u_{\omega}^s(D,z;x)] &= \mathring{u}_{p}^s(z;x)[\hat{x}^Td_{0}\mO(\frac{1}{\lambda}) \hat{x} +
\mathcal{O}(|z|^{-1}) + \mO(\omega)][1+ \mathcal{O}(|z|^{-1})], \label{eq:relation:near:far:p}
\end{align}
uniformly for $\hat{z}\in \mathbb{S}^2$ and $x\in S$. Substituting  \eqref{eq:relation:near:far:p} into \eqref{eq:indicator:p}, we arrive at \eqref{eq:indica:relation:p}.

By the Cauchy-Schwarz inequality, we see $ \mathring{I}_{p}(z;\tilde{z})  \leq 1$ unless $\mathring{u}_{p}^s(z;\cdot)$ and $\mathring{u}_{p}^s(\tilde{z};\cdot)$ are constant multiples of each other, which can only happen when $z=\tilde{z}$. This then helps us to find the location $z$ with the sampling points $\tilde{z}$.

\end{proof}

\subsection{Shape determination}
Once the location is approximated by $\mathring{z}= \argmax_{\tilde{z}}I_{p}(D,z;\tilde{z})$ or $\mathring{z}= \argmax_{\tilde{z}}I_{s}(D,z;\tilde{z})$ in Section \ref{subsect:locating}, together with Corollary \ref{cor:near:shape:rigid:dete} and \ref{cor:asym:medium:shape:dete}, we next present the numerical schemes for reconstructing the shapes through dictionary matching under a regular frequency. The corresponding analysis also depends on the noise level $\epsilon$ through a similar strategy as before.

\medskip

Case I:~The noise level $\epsilon \sim \smallO (\frac{1}{\lambda})$.
In view of Corollaries \ref{cor:near:shape:rigid:dete} and \ref{cor:asym:medium:shape:dete}, we propose the following imaging functionals,
\begin{equation}\label{eq:indicator:p:multi}
J_{p}(D_i, D_j; z, \mathring{z}): = \frac{|\langle P_{\hat{x}}[u_{\omega}^s(D_i,z;\cdot)], P_{\hat{x}}[u_{sp}^s(D_j,\mathring{z};\cdot)]\rangle_{L^{2}(\Lambda)} |}{\|P_{\hat{x}}[u_{\omega}^s(D_i,z;\cdot)]\|_{L^2(\Lambda)} \|P_{\hat{x}}[u_{sp}^s(D_j,\mathring{z};\cdot)]\|_{L^2(\Lambda)} },
\end{equation}
for $D_i, D_j \in \mathfrak{D}$, where $u_{\omega}^s(D_i,z;\cdot): = u_{\omega, D_i+z}^{s}(\cdot)$, and the test function
\begin{equation}\label{eq:shape:sp:p}
u_{sp}^s(D_j,\mathring{z};x): =  \frac{e^{ik_s|\mathring{z}|}}{4\pi |\mathring{z}|} \frac{e^{ik_p|x-\mathring{z}|}}{|x-\mathring{z}|}u_{s}^{\infty}(D_j,\hat{\mathring{z}},p;\widehat{x-\mathring{z}}),
\end{equation}
with $u_{s}^{\infty}(D_j,\hat{\mathring{z}},p;\widehat{x-\mathring{z}})$ the far field of the scattered wave of $D$ ($D$ is a rigid obstacle or a penetrable medium) due to the incident shear plane wave only, i.e., $u_{s}^i$ in \eqref{eq:plane}.
%

\medskip
Case II: The noise level $ \epsilon \sim \mO(\frac{1}{\lambda})$.
 We can again use near-field scattered data in the dictionary produced by the shear incident plane wave $u_{s}^i$. We denote it as $u_{ss}^s$, i.e.,
 \begin{equation}
u_{ss}^s(D_j,\mathring{z};x): =  \frac{e^{ik_s|\mathring{z}|}}{4\pi |\mathring{z}|} \frac{e^{ik_s|x-\mathring{z}|}}{|x-\mathring{z}|}u_{s}^{\infty}(D_j,\hat{\mathring{z}},p;\widehat{x-\mathring{z}}).
\end{equation}
 Here $u_{s}^{\infty}(D_j,\hat{\mathring{z}},p;\widehat{x-\mathring{z}})$ is the same as in \eqref{eq:shape:sp:p}. Then the imaging functional becomes

\begin{equation}\label{eq:indicator:s:multi}
J_{s}(D_i, D_j;z, \mathring{z}): = \frac{|\langle (I - P_{\hat{x}}) u_{\omega}^{s}(D_i,z;\hat{x}),  (I - P_{\hat{x}})\mathring{u}_{ss}^{s}(D_j,\mathring{z};\hat{x})\rangle_{L^2(\Lambda)} | }{\|(I - P_{\hat{x}})u_{\omega}^{s}(D_i,z;\hat{x})\|_{L^2(\Lambda)} \|(I - P_{\hat{x}})\mathring{u}_{ss}^{s}(D_j,\mathring{z};\hat{x})\|_{L^2(\Lambda)} }
\end{equation}

With these preparations, we can present our second stage algorithm for the shape reconstruction. Through this scheme, we can find the shape of the scatterer $\Omega$ by dictionary searching.

\begin{theorem}
 Assume that there exists a constant $c_0>0$ such that $\|P_{\hat{x}}[u_{sp}^s(D_j,\mathring{z};\cdot)]\|_{L^2(\Lambda)} \geq c_0$ (or $\|(I - P_{\hat{x}})\mathring{u}_{ss}^{s}(D_j,\mathring{z};\hat{x})\|_{L^2(\Lambda)} \geq c_0$) for all $D_i \in \mathfrak{D}$. Then for any sufficiently small $\varepsilon>0$ there exist $R_0$ and $\delta>0$ such that
if $|z| \geq R_0$ and $|z-\mathring{z}| \leq \delta$,
\begin{equation}\label{eq:shape:appro}
|J_{l_i}(D_i,D_j;z, \mathring{z}) - \hat{J}_{l_i}(D_i,D_j;z)| \leq z, \quad \forall D_i, D_j \in \mathfrak{D},
\end{equation}
where
\[
\hat{J}_{l_i}(D_i,D_j;z): = J_{l_i}(D_i,D_j;z,z), \quad  l_i = p \ \  \text{or} \  \  l_i=s.
\]
Furthermore, as $\mathring{z} \rightarrow z$ and $D_i = D_j$, $J_{l_i}(D_i,D_j;z, \mathring{z})$ obtains its maximal value $1$ approximately.
\end{theorem}
\begin{proof}
We first consider the case that $D$ is a rigid obstacle. For the indicator $J_{p}$ with the noise level $\epsilon \sim  \smallO (\frac{1}{\lambda})$, it suffices to calculate $P_{\hat{x}}[u_{\omega}^s(D_i,z;\cdot)]$ and  $P_{\hat{x}}[u_{sp}^s(D_j,\mathring{z};\cdot)]$ as in \eqref{eq:indicator:p:multi}. According to Corollary \ref{cor:near:shape:rigid:dete}, we have
\begin{align*}
P_{\hat{x}}[u_{\omega}^s(D_i,z;\cdot)] = P_{\hat{x}} [u_{\omega,\Omega}^s(x)] =& \frac{e^{ik_s|z|}}{4 \pi |z|}  \frac{e^{ik_p|x-z|}}{|x-z|}
(\mathcal{F}_{p,\infty}\phi_{s}^\omega)(\widehat{x-z}) + \mO(\frac{1}{|z|^2})  \\
 &   +\frac{e^{ik_p|z|}}{4 \pi |z|}   \frac{e^{ik_p|x-z|}}{|x-z|} (\mathcal{F}_{p,\infty}\phi_{p}^\omega)(\widehat{x-z}) + \mathcal{O}(\frac{1}{|x|^2}).
\end{align*}
Since $|(\mathcal{F}_{p,\infty}\phi_{p}^\omega)(\widehat{x-z})| = \mO(\frac{1}{\lambda^2})$ and $|(\mathcal{F}_{p,\infty}\phi_{s}^\omega)(\widehat{x-z})| = \mO(\frac{1}{\lambda})$, we have
\begin{equation}\label{eq:shape:unknown:p}
P_{\hat{x}}[u_{\omega}^s(D_i,z;\cdot)] = \frac{e^{ik_s|z|}}{4 \pi |z|}   \frac{e^{ik_p|x-z|}}{|x-z|} (\mathcal{F}_{p,\infty}\phi_{s}^\omega)(\widehat{x-z}) + \mathcal{O}(\frac{1}{|x|^2}) + \mO(\frac{1}{\lambda^2}).
\end{equation}
Moreover, for the precomputed dictionary data $u_{sp}^s(D_j,\mathring{z};\cdot)$, we have
\begin{equation}\label{eq:shape:known:p}
P_{\hat{x}}[u_{sp}^s(D_j,\mathring{z};\cdot)] = \frac{e^{ik_s|\mathring{z}|}}{4 \pi |\mathring{z}|}   \frac{e^{ik_p|x-\mathring{z}|}}{|x-\mathring{z}|} (\mathcal{F}_{p,\infty}\phi_{s}^\omega)(\widehat{x-\mathring{z}}) + \mathcal{O}(\frac{1}{|x|^2}).
\end{equation}
It can be seen that the dictionary data \eqref{eq:shape:known:p} approximate the measured scattered field \eqref{eq:shape:unknown:p} under the condition of this theorem, and then the approximation \eqref{eq:shape:appro} follows.
By the Cauchy-Schwarz inequality and comparing the above two formulas, we see that $J_{p}(D_i,D_j;z, \mathring{z})$ achieves its maximal value approximately if and only if $D_i = D_j$.

For the indicator $J_s$ with the noise level $ \epsilon \sim \mO(\frac{1}{\lambda})$, it is sufficient for us to calculate $(I - P_{\hat{x}}) u_{\omega}^{s}(D,z;\hat{x})$ and $(I - P_{\hat{x}})\mathring{u}_{ss}^{s}(D_j,\mathring{z};\hat{x})$. By Corollary \ref{cor:near:shape:rigid:dete}, we have
\begin{align*}
(I - P_{\hat{x}})u_{\omega}^{s}(D,z;\hat{x}) =& (I - P_{\hat{x}}) u_{\omega,\Omega}^s(x)
 = \frac{e^{ik_s|z|}}{4 \pi |z|}  \frac{e^{ik_s|x-z|}}{|x-z|} (\mathcal{F}_{s,\infty}\phi_{s}^\omega)(\widehat{x-z}) + \mathcal{O}(\frac{1}{|x|^2}) \\
 &+\frac{e^{ik_p|z|}}{4 \pi |z|}  \frac{e^{ik_s|x-z|}}{|x-z|}  (\mathcal{F}_{s,\infty}\phi_{p}^\omega)(\widehat{x-z}) + \mO(\frac{1}{|z|^2}).
\end{align*}
Since $|(\mathcal{F}_{s,\infty}\phi_{s}^\omega)(\widehat{x-z}) |= \mO(1)$ and $|(\mathcal{F}_{s,\infty}\phi_{p}^\omega)(\widehat{x-z})| = \mO(\frac{1}{\lambda})$, we thus have
\begin{equation}\label{eq:shape:unknown}
(I - P_{\hat{x}})u_{\omega}^{s}(D,z;\hat{x})  = \frac{e^{ik_s|z|}}{4 \pi |z|} \frac{e^{ik_s|x-z|}}{|x-z|} (\mathcal{F}_{s,\infty}\phi_{s}^\omega)(\widehat{x-z}) + \mathcal{O}(\frac{1}{|x|^2}) + \mO(\frac{1}{\lambda}) +  \mathcal{O}(\frac{1}{|z|^2}).
\end{equation}
Similarly,  it can be verified that
\begin{equation}\label{eq:shape:known}
(I - P_{\hat{x}})\mathring{u}_{ss}^{s}(D_j,\mathring{z};\hat{x}) = \frac{e^{ik_s|z|}}{4 \pi |z|}  \frac{e^{ik_s|x-z|}}{|x-z|} (\mathcal{F}_{s,\infty}\phi_{s}^\omega)(\widehat{x-z}) + \mathcal{O}(\frac{1}{|x|^2}).
\end{equation}
Hence the dictionary data \eqref{eq:shape:known} approximate the measured scattered field \eqref{eq:shape:unknown} under the condition of this theorem, and the approximation \eqref{eq:shape:appro} also follows.
Again by Cauchy-Schwarz inequality and comparing the above two formulas, we conclude that $J_{s}(D_i,D_j;z, \mathring{z})$ achieves its maximal value approximately if and only if $D_i = D_j$.

The other case that $\Omega$ is a penetrable medium follows by completely similar arguments, along with the use of Corollary \ref{cor:asym:medium:shape:dete}.
\end{proof}

\section{Numerical Tests}\label{sec:num}

In this section, we present numerical tests to illustrate the effectiveness and efficiency of the proposed method.
We verify the elastic imaging technique using limited aperture near field data. 
All the numerical experiments are carried out using MATLAB R2017a on a
Lenovo workstation with 2.3GHz Intel Xeon E5-2670 v3 processor and 512GB of RAM.

The experimental setup is as follows.
As shown in Figure~\ref{fig:Dictionary}, the admissible dictionary set consists
of six referemce domains, $D_i$, $i=1,\ldots,6$, which are composed of a number of unit cubes.
The measurement surface $\Lambda$ is set to be a unit square in the $x^2 x^3$-plane and centered at the origin. The scattered elastic near fields on the measurement surface $\Lambda$ of reference domains
in the dictionary  $\mathfrak{D}$ as in \eqref{eq:dic:class} are first collected in advance for incident point signal waves with different locations.

In the following tests, we let the Poisson ratio $\nu=0.475$ and Young's modulus $E=3$, thus $\mu \sim \mO (1)$,  $\lambda \sim  \mO (10^2)$ and $\lambda \gg \mu$.
The detecting wavelength for locating objects is set
to be $\omega_{1}:=1$ and the detecting wavelength for shape determination
is set to be $\omega_{2}:=20$.
Without loss of generality, the target scatterer is given by $\Omega:=D_{i}+z_{0}$.
Here $z_{0}$ is fixed as $(40,\,0,\,0)$.

\begin{figure}
\hfill{}\includegraphics[width=0.3\textwidth]{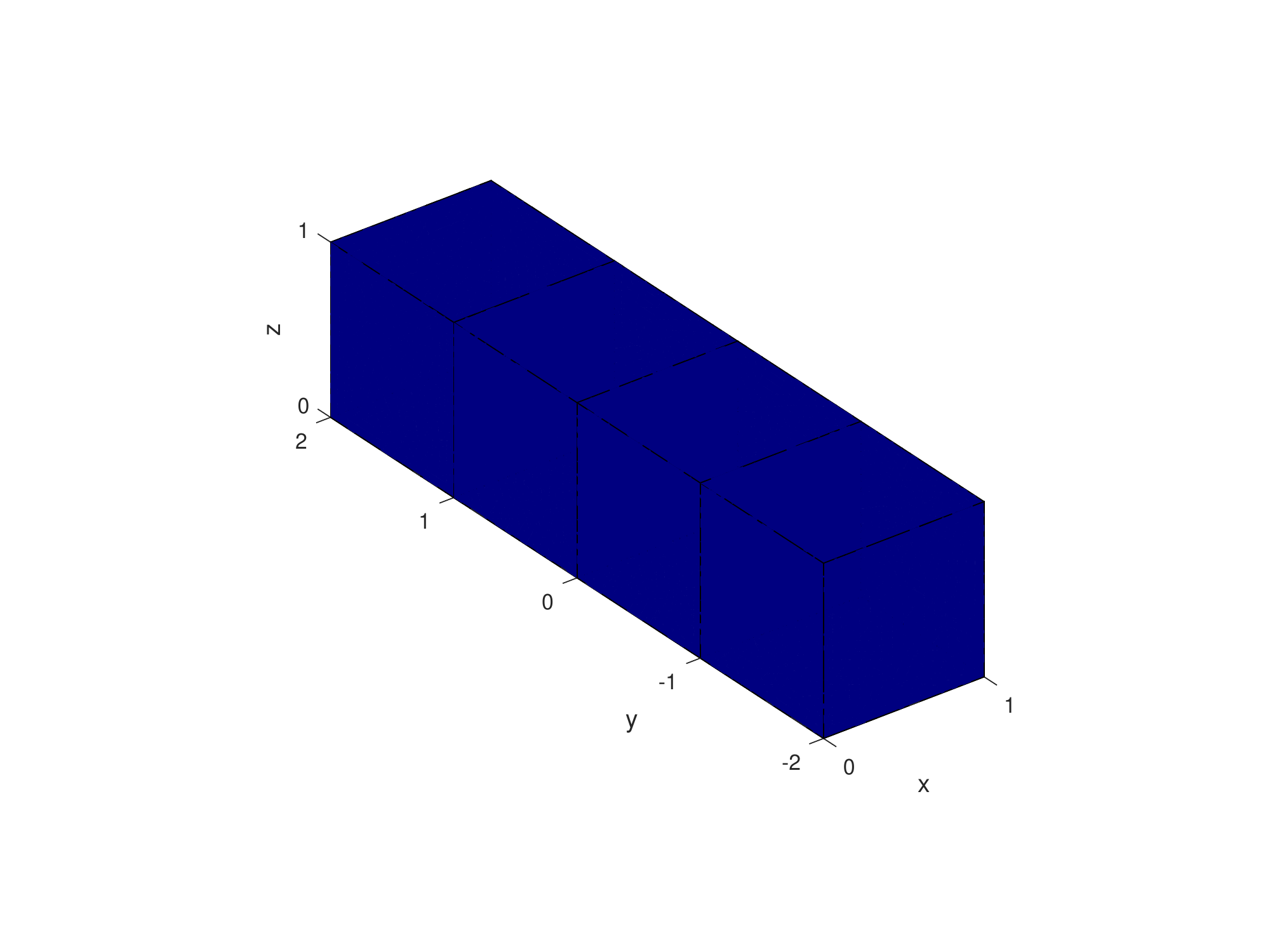}\hfill{}\includegraphics[width=0.3\textwidth]{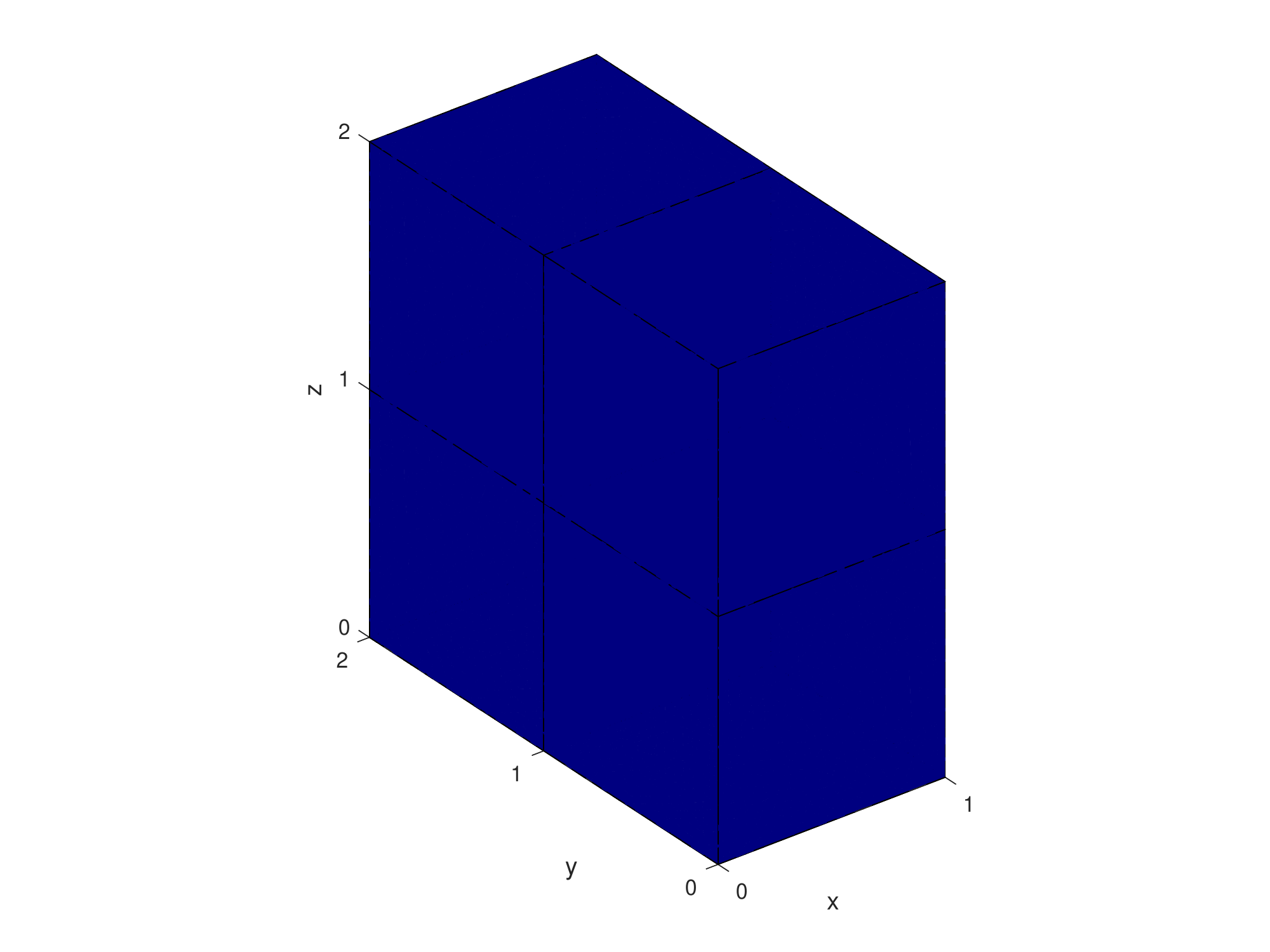}\hfill{}\includegraphics[width=0.3\textwidth]{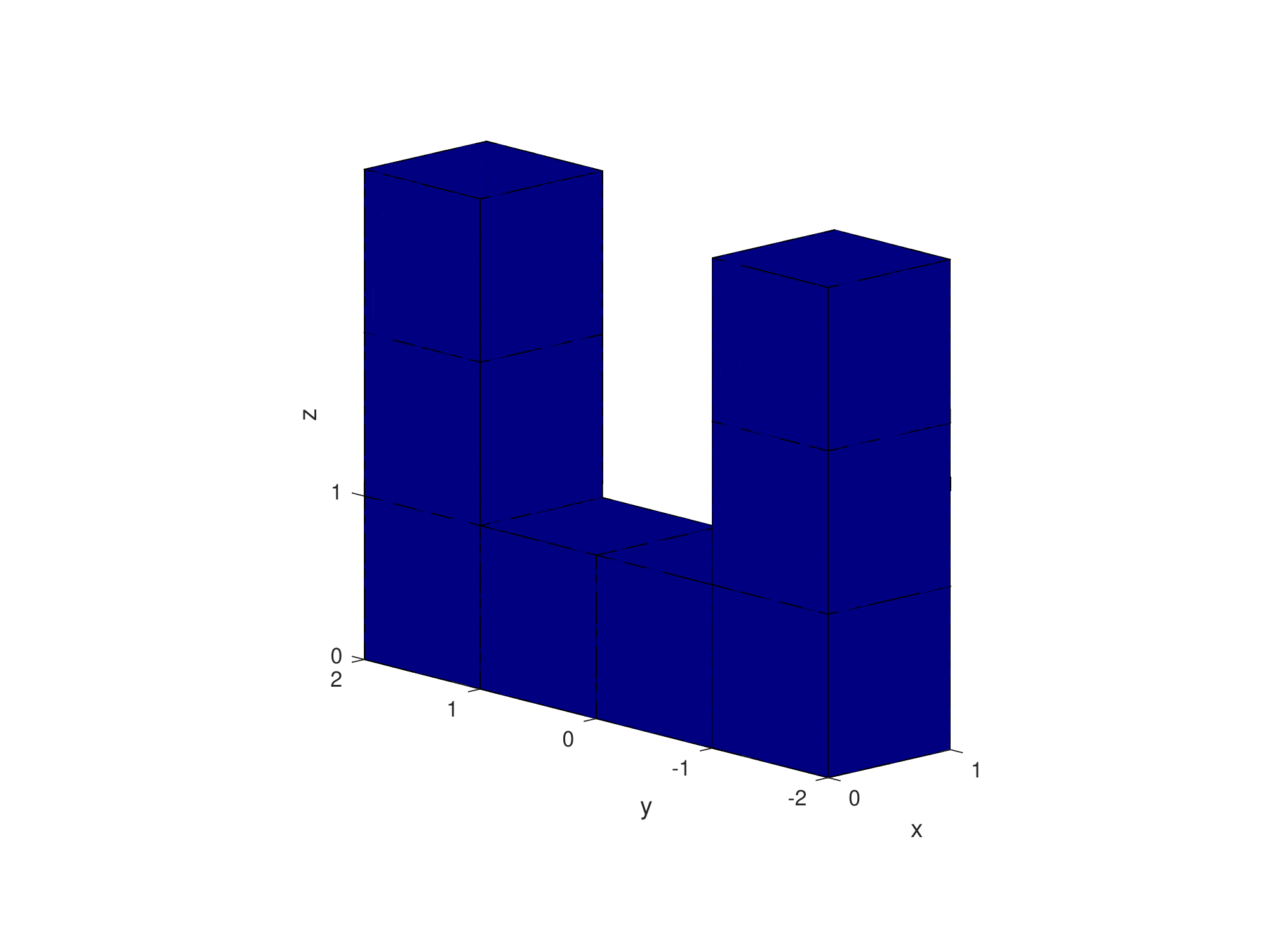}\hfill{}

\hfill{}(a)\hfill{}(b)\hfill{}(c)\hfill{}

\hfill{}\includegraphics[width=0.3\textwidth]{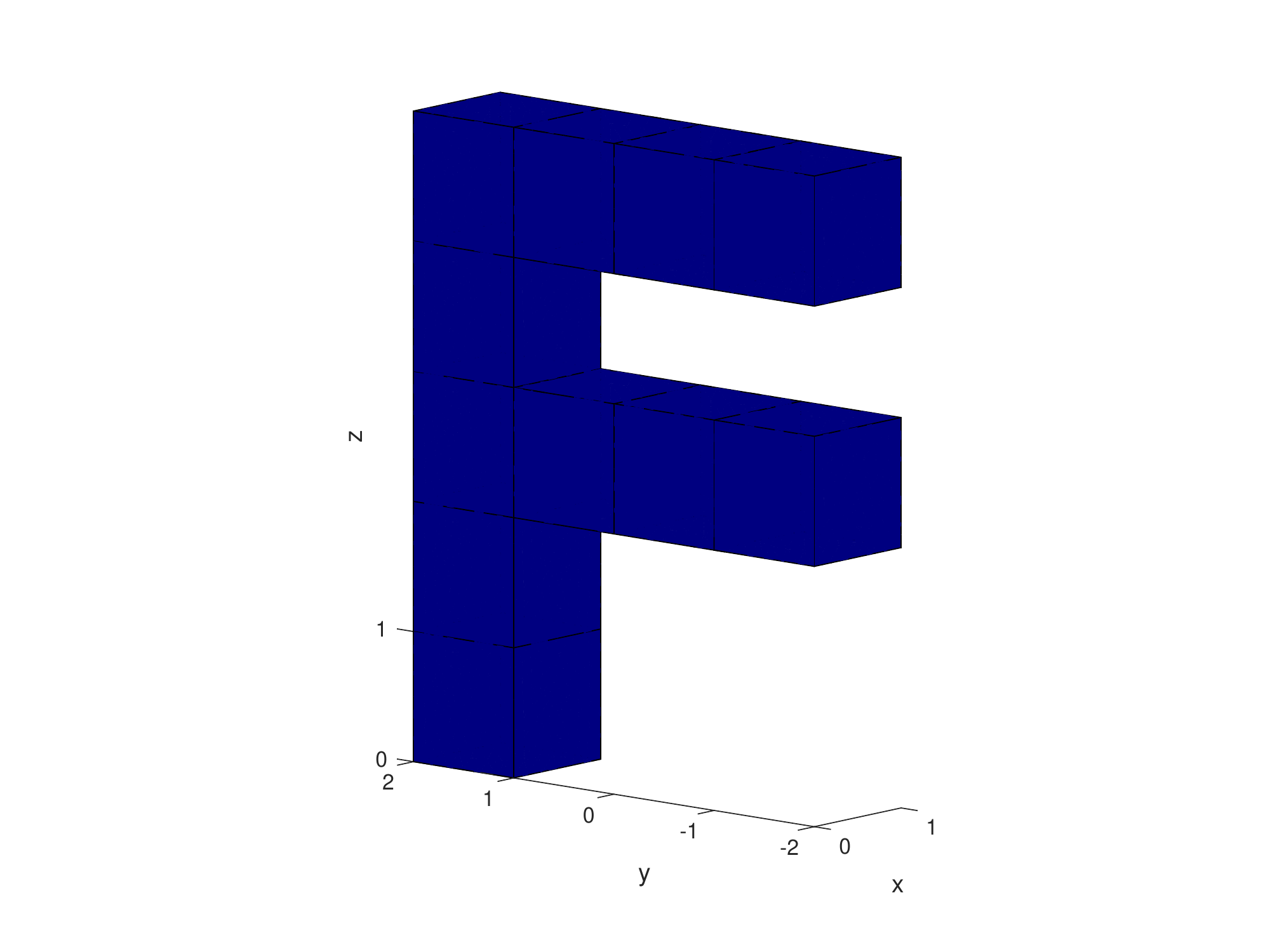}\hfill{}\includegraphics[width=0.3\textwidth]{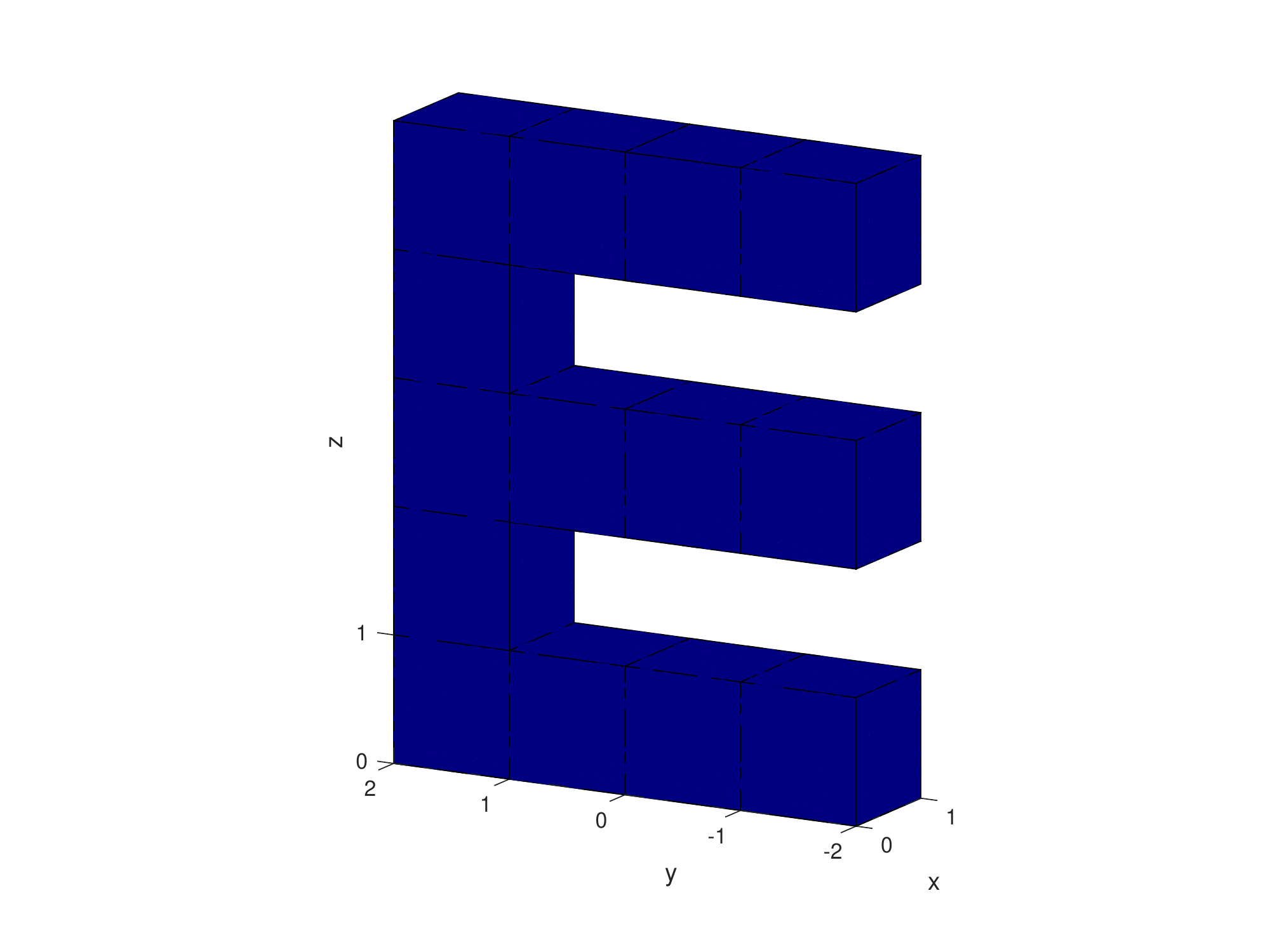}\hfill{}\includegraphics[width=0.3\textwidth]{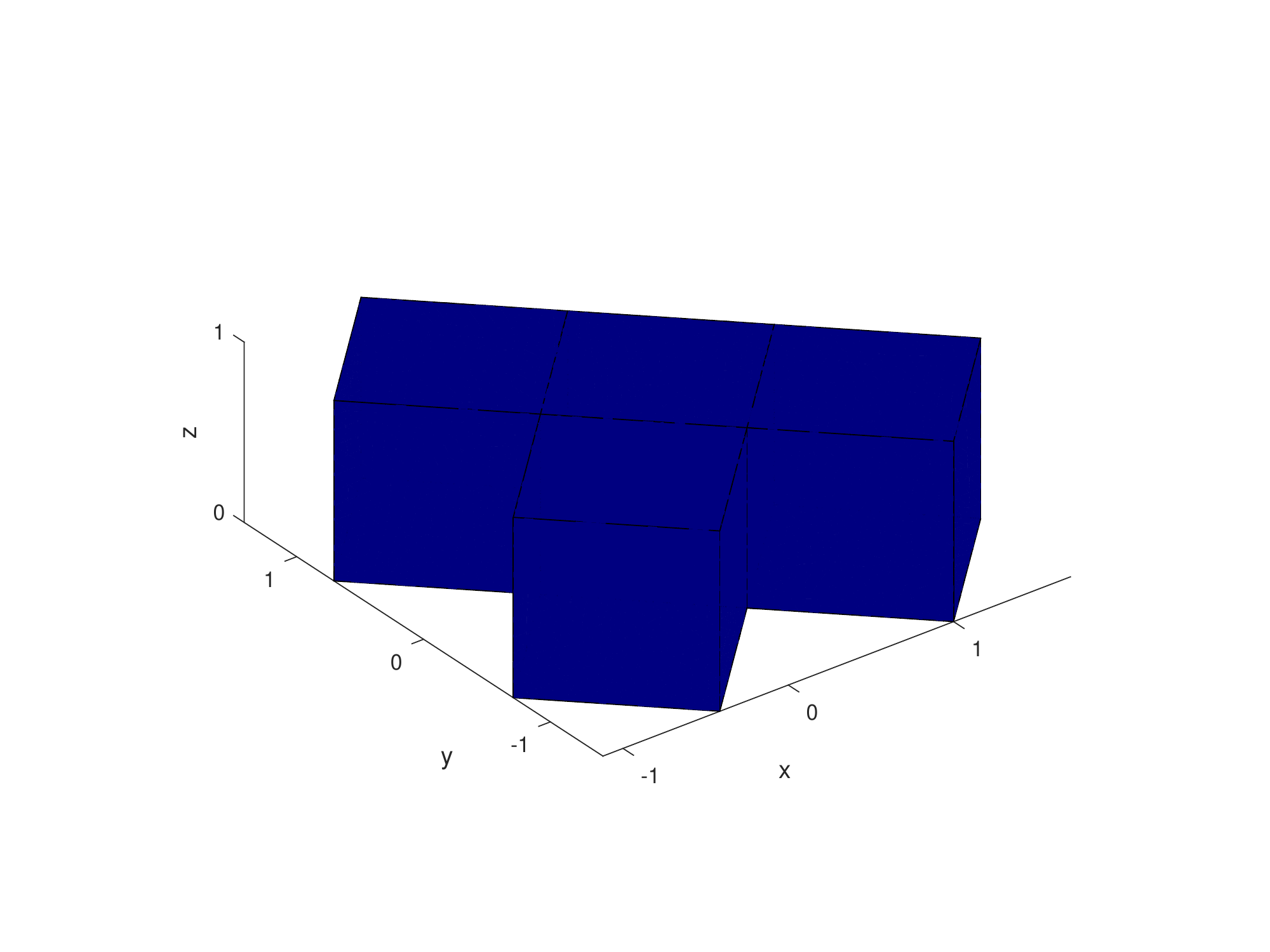}\hfill{}

\hfill{}(d)\hfill{}(e)\hfill{}(f)\hfill{}

\caption{\label{fig:Dictionary}Dictionary of a priori known scatterers.}
\end{figure}

\subsection{Rigid body case}

In the first example, we test with  six reference domains with the rigid body boundary condition.

In the first stage, the positions are found in the locating stage by finding the maximum value of the indicator functions $I_p$ in \eqref{eq:indicator:p} and $I_s$ in \eqref{eq:indicator:s}.
In the noise-free case,   the coordinates and
the distance from the exact location are shown in Table~\ref{tab:location-test-pec} for each reference domain in the dictionary.
One can find that the error of positions deduced from $I_{s}$ is much
smaller than those from $I_{p}$. The reason for this phenomenon is due to the fact that $k_{s}\gg k_{p}$,
thus the shear wave with wave number $k_{s}$ has a higher resolution. In light of the high resolution of shear waves, it is enough for us to locate and imaging the reference domain only
by $I_{s}$ in \eqref{eq:indicator:s}.
Compared with the exact position $z_0$, the difference between the exact and estimated positions using $I_s$ in \eqref{eq:indicator:s} are always below $0.1\%$ in terms of the Euclidean distance.

Next, we compute indicator function value $J_s$ in \eqref{eq:indicator:s:multi} using the near field data measured on $\Lambda$ and the approximate position found in Table~\ref{tab:location-test-pec}.
The result of shape determination is shown in Table~\ref{tab:gesture-test-pec}. The values of the indicator function value have  been
rescaled between $0$ and $1$ by normalizing with respect to the maximum function value among all six reference gestures in each row of the table to highlight the unique reference domain
identified. The same normalization procedure is employed in the sequel. We see from Table~\ref{tab:gesture-test-pec} that the peak value are always taken in the diagonal line
when the measurement data match with the precomputed data of the correct gesture.

\begin{table}

\hfill{}%
\begin{tabular}{|c|c|c|c|c|c|c|}
\hline
 & $D_{1}$ & $D_{2}$ & $D_{3}$ & $D_{4}$ & $D_{5}$ & $D_{6}$\tabularnewline
\hline
\hline
$\mathring{z}_{0}^{1}$ & $40.0278$ & $40.0547$ & $40.0965$ & $40.0158$ & $40.0971$ & $40.0958$\tabularnewline
\hline
$\mathring{z}_{0}^{2}$ & $0.0679$ & $0.0743$ & $0.0655$ & $0.0706$ & $0.0277$ & $0.0097$\tabularnewline
\hline
$\mathring{z}_{0}^{3}$ & $0.0758$ & $0.0392$ & $0.0171$ & $0.0032$ & $0.0046$ & $0.0823$\tabularnewline
\hline
$\left|\mathring{z}_{0}-z_{0}\right|$ & $0.1055$ & $0.1003$ & $0.1173$ & $0.1196$ & $0.0322$ & $0.1277$\tabularnewline
\hline
\end{tabular}\hfill{}

\smallskip

\hfill{}$I_{p}$  indicator function.\hfill{}

\medskip

\hfill{}%
\begin{tabular}{|c|c|c|c|c|c|c|}
\hline
 & $D_{1}$ & $D_{2}$ & $D_{3}$ & $D_{4}$ & $D_{5}$ & $D_{6}$\tabularnewline
\hline
\hline
$\mathring{z}_{0}^{1}$ & $39.9901$ & $39.9811$ & $39.9784$ & $39.9846$ & $40.0035$ & $39.9853$\tabularnewline
\hline
$\mathring{z}_{0}^{2}$ & $0.0100$ & $0.0180$ & $-0.0056$ & $0.0201$ & $-0.0251$ & $0.0121$\tabularnewline
\hline
$\mathring{z}_{0}^{3}$ & $-0.0101$ & $-0.0110$ & $-0.0204$ & $-0.0256$ & $-0.0265$ & $-0.0265$\tabularnewline
\hline
$\left|\mathring{z}_{0}-z_{0}\right|$ & $0.0173$ & $0.0283$ & $0.0302$ & $0.0360$ & $0.0367$ & $0.0326$\tabularnewline
\hline
\end{tabular}\hfill{}

\smallskip

\hfill{}$I_{s}$  indicator function.\hfill{}

\caption{\label{tab:location-test-pec} Location test using $I_{p}$ and $I_{s}$  for elastic  rigid bodies without noise. }
\end{table}
\begin{table}
\hfill{}%
\begin{tabular}{|c|c|c|c|c|c|c|}
\hline
 & $D_{1}$ & $D_{2}$ & $D_{3}$ & $D_{4}$ & $D_{5}$ & $D_{6}$\tabularnewline
\hline
\hline
$D_{1}$ & $\boldsymbol{1.0000}$ & $0.9234$ & $0.9291$ & $0.8660$ & $0.8249$ & $0.9453$\tabularnewline
\hline
$D_{2}$ & $0.9232$ & $\boldsymbol{1.0000}$ & $0.9245$ & $0.9502$ & $0.9109$ & $0.9146$\tabularnewline
\hline
$D_{3}$ & $0.9280$ & $0.9242$ & $\boldsymbol{1.0000}$ & $0.9040$ & $0.9494$ & $0.9851$\tabularnewline
\hline
$D_{4}$ & $0.8650$ & $0.9510$ & $0.9040$ & $\boldsymbol{1.0000}$ & $0.9301$ & $0.9742$\tabularnewline
\hline
$D_{5}$ & $0.8249$ & $0.9007$ & $0.9592$ & $0.9310$ & $\boldsymbol{1.0000}$ & $0.9169$\tabularnewline
\hline
$D_{6}$ & $0.9451$ & $0.9147$ & $0.9849$ & $0.9732$ & $0.9249$ & $\boldsymbol{1.0000}$\tabularnewline
\hline
\end{tabular}\hfill{}

\smallskip

\hfill{}$J_{p} $  indicator function.\hfill{}

\medskip

\hfill{}%
\begin{tabular}{|c|c|c|c|c|c|c|}
\hline
 & $D_{1}$ & $D_{2}$ & $D_{3}$ & $D_{4}$ & $D_{5}$ & $D_{6}$\tabularnewline
\hline
\hline
$D_{1}$ & $\boldsymbol{1.0000}$ & $0.9134$ & $0.8291$ & $0.8260$ & $0.8249$ & $0.9453$\tabularnewline
\hline
$D_{2}$ & $0.9212$ & $\boldsymbol{1.0000}$ & $0.9245$ & $0.9202$ & $0.9309$ & $0.9146$\tabularnewline
\hline
$D_{3}$ & $0.9260$ & $0.9242$ & $\boldsymbol{1.0000}$ & $0.9240$ & $0.9194$ & $0.9351$\tabularnewline
\hline
$D_{4}$ & $0.8630$ & $0.8510$ & $0.9140$ & $\boldsymbol{1.0000}$ & $0.9101$ & $0.9542$\tabularnewline
\hline
$D_{5}$ & $0.8149$ & $0.8007$ & $0.9492$ & $0.9310$ & $\boldsymbol{1.0000}$ & $0.9269$\tabularnewline
\hline
$D_{6}$ & $0.9461$ & $0.8147$ & $0.9449$ & $0.9532$ & $0.9249$ & $\boldsymbol{1.0000}$\tabularnewline
\hline
\end{tabular}\hfill{}

\smallskip

\hfill{}$J_{s}$  indicator function. \hfill{}

\caption{\label{tab:gesture-test-pec}Shape determination test using $J_{p}$ and $J_{s}$  for elastic rigid body without noise.}
\end{table}

In the noisy case with noise level of $5\%$,  the positions found in the first stage are shown in Table \ref{tab:location-test-pec-noise}.
the difference between the exact and estimated positions is still very small.
The result of gesture recognition is shown in Table \ref{tab:gesture-test-pec-noise}, which clearly shows that  all the correct pairs matches the best.
The test with noisy data shows the robustness with respect to
noisy measurement data  of both the locating and recognition indicator functions $I_s$ in \eqref{eq:indicator:s} and $J_s$ in \eqref{eq:indicator:s:multi}, respectively. This salient robustness
is due to the inner product operation, which eliminates implicitly the noisy part in light of the orthogonality.

%
%

\begin{table}
\hfill{}%
\begin{tabular}{|c|c|c|c|c|c|c|}
\hline
 & $D_{1}$ & $D_{2}$ & $D_{3}$ & $D_{4}$ & $D_{5}$ & $D_{6}$\tabularnewline
\hline
\hline
$\mathring{z}_{0}^{1}$                  & $39.8738$ & $39.7611$ & $39.8742$ & $39.8632$ & $40.1075$ & $39.8453$\tabularnewline
\hline
$\mathring{z}_{0}^{2}$                  & $0.1307$  & $0.3280$  & $-0.0896$ & $0.1542$  & $-0.1451$ & $0.1321$\tabularnewline
\hline
$\mathring{z}_{0}^{3}$                  & $-0.1731$ & $-0.2110$ & $-0.1204$ & $-0.1456$ & $-0.1695$ & $-0.1475$\tabularnewline
\hline
$\left|\mathring{z}_{0}-z_{0}\right|$   & $0.2509$  & $0.4574$  & $0.1958$  & $0.2524$  & $0.2477$  & $0.2513$\tabularnewline
\hline
\end{tabular}\hfill{}

\caption{\label{tab:location-test-pec-noise}Location test using $I_{s}$  for elastic  rigid body with noise $5\%$.}
\end{table}
\begin{table}
\hfill{}%
\begin{tabular}{|c|c|c|c|c|c|c|}
\hline
 & $D_{1}$ & $D_{2}$ & $D_{3}$ & $D_{4}$ & $D_{5}$ & $D_{6}$\tabularnewline
\hline
\hline
$D_{1}$ & $\boldsymbol{1.0000}$ & $0.9453$ & $0.9632$ & $0.8213$ & $0.9182$ & $0.9649$\tabularnewline
\hline
$D_{2}$ & $0.9431$ & $\boldsymbol{1.0000}$ & $0.9374$ & $0.8864$ & $0.9205$ & $0.9061$\tabularnewline
\hline
$D_{3}$ & $0.9651$ & $0.9255$ & $\boldsymbol{1.0000}$ & $0.9213$ & $0.9070$ & $0.9124$\tabularnewline
\hline
$D_{4}$ & $0.8268$ & $0.8811$ & $0.9219$ & $\boldsymbol{1.0000}$ & $0.9621$ & $0.9450$\tabularnewline
\hline
$D_{5}$ & $0.9152$ & $0.9213$ & $0.9071$ & $0.9491$ & $\boldsymbol{1.0000}$ & $0.9378$\tabularnewline
\hline
$D_{6}$ & $0.9649$ & $0.9066$ & $0.9123$ & $0.9459$ & $0.9367$ & $\boldsymbol{1.0000}$\tabularnewline
\hline
\end{tabular}\hfill{}

\caption{\label{tab:gesture-test-pec-noise}Shape determination test using $J_{s}$  for elastic rigid body with noise $5\%$.}
\end{table}

\subsection{Medium case}

In the second example, we test with an inhomogeneous elastic medium among the six reference gesture domains.

\[
n_{k,\,\Omega}=\begin{cases}
1 & x\in\mathbb{R}^{3}\backslash\bar{\Omega}\\
5 & x\in\Omega
\end{cases}.
\]

The results of location and shape determination tests are shown, respectively, in Tables~\ref{tab:location-test-medium} and \ref{tab:gesture-test-medium} for the noise-free case,
and in Tables~\ref{tab:location-test-medium-noise} and \ref{tab:gesture-test-medium-noise} for the noisy case with $5\%$ noise level.
Both noise-free and noisy cases tell us that our locating and shape determination algorithms are very robust with noise and work very well even with data of limited aperture.
Besides, the computational efforts is quite less and the shape determination schemes are very efficient only involving with inner product by known data at hand.

\begin{table}
\hfill{}%
\begin{tabular}{|c|c|c|c|c|c|c|}
\hline
 & $D_{1}$ & $D_{2}$ & $D_{3}$ & $D_{4}$ & $D_{5}$ & $D_{6}$\tabularnewline
\hline
\hline
$\mathring{z}_{0}^{1}$ & $40.0417$ & $40.0254$ & $39.9576$ & $40.0279$ & $40.0069$ & $39.9837$\tabularnewline
\hline
$\mathring{z}_{0}^{2}$ & $-0.0214$ & $-0.0120$ & $-0.0446$ & $0.0434$ & $-0.0031$ & $-0.0338$\tabularnewline
\hline
$\mathring{z}_{0}^{3}$ & $0.0257$ & $0.0068$ & $0.0031$ & $-0.0370$ & $-0.0488$ & $0.0294$\tabularnewline
\hline
$\left|\mathring{z}_{0}-z_{0}\right|$ & $0.0535$ & $0.0289$ & $0.0616$ & $0.0635$ & $0.0494$ & $0.0477$\tabularnewline
\hline
\end{tabular}\hfill{}

\caption{\label{tab:location-test-medium}Location test using $I_{s}$  for elastic media without noise. }
\end{table}
\begin{table}
\hfill{}%
\begin{tabular}{|c|c|c|c|c|c|c|}
\hline
 & $D_{1}$ & $D_{2}$ & $D_{3}$ & $D_{4}$ & $D_{5}$ & $D_{6}$\tabularnewline
\hline
\hline
$D_{1}$ & $\boldsymbol{1.0000}$ & $0.9311$ & $0.8848$ & $0.9393$ & $0.8716$ & $0.9418$\tabularnewline
\hline
$D_{2}$ & $0.9312$ & $\boldsymbol{1.0000}$ & $0.9174$ & $0.8838$ & $0.9100$ & $0.9086$\tabularnewline
\hline
$D_{3}$ & $0.8834$ & $0.9179$ & $\boldsymbol{1.0000}$ & $0.9295$ & $0.9740$ & $0.8854$\tabularnewline
\hline
$D_{4}$ & $0.9398$ & $0.8899$ & $0.9004$ & $\boldsymbol{1.0000}$ & $0.9200$ & $0.9864$\tabularnewline
\hline
$D_{5}$ & $0.8737$ & $0.9171$ & $0.9422$ & $0.9183$ & $\boldsymbol{1.0000}$ & $0.9131$\tabularnewline
\hline
$D_{6}$ & $0.9346$ & $0.9087$ & $0.8557$ & $0.9131$ & $0.9089$ & $\boldsymbol{1.0000}$\tabularnewline
\hline
\end{tabular}\hfill{}

\caption{\label{tab:gesture-test-medium}Shape determination test using $J_{s}$  for elastic media without noise.}
\end{table}

\begin{table}
\hfill{}%
\begin{tabular}{|c|c|c|c|c|c|c|}
\hline
 & $D_{1}$ & $D_{2}$ & $D_{3}$ & $D_{4}$ & $D_{5}$ & $D_{6}$\tabularnewline
\hline
\hline
$\mathring{z}_{0}^{1}$ & $39.9758$ & $39.9762$ & $39.9722$ & $39.9819$ & $39.9586$ & $39.9529$\tabularnewline
\hline
$\mathring{z}_{0}^{2}$ & $-0.0091$ & $0.0103$ & $-0.0383$ & $-0.0076$ & $-0.0238$ & $0.0429$\tabularnewline
\hline
$\mathring{z}_{0}^{3}$ & $0.0095$ & $0.0211$ & $-0.0203$ & $0.0008$ & $0.0301$ & $0.0230$\tabularnewline
\hline
$\left|\mathring{z}_{0}-z_{0}\right|$ & $0.0275$ & $0.0334$ & $0.0515$ & $0.0197$ & $0.0565$ & $0.0677$\tabularnewline
\hline
\end{tabular}\hfill{}

\caption{\label{tab:location-test-medium-noise}Location test using $I_{s}$  for elastic media with noise $5\%$.}
\end{table}
\begin{table}
\hfill{}%
\begin{tabular}{|c|c|c|c|c|c|c|}
\hline
 & $D_{1}$ & $D_{2}$ & $D_{3}$ & $D_{4}$ & $D_{5}$ & $D_{6}$\tabularnewline
\hline
\hline
$D_{1}$ & $\boldsymbol{1.0000}$ & $0.8800$ & $0.9109$ & $0.9321$ & $0.8895$ & $0.9280$\tabularnewline
\hline
$D_{2}$ & $0.8738$ & $\boldsymbol{1.0000}$ & $0.8666$ & $0.9520$ & $0.9303$ & $0.9032$\tabularnewline
\hline
$D_{3}$ & $0.9105$ & $0.8656$ & $\boldsymbol{1.0000}$ & $0.9095$ & $0.8817$ & $0.9021$\tabularnewline
\hline
$D_{4}$ & $0.9320$ & $0.9221$ & $0.9802$ & $\boldsymbol{1.0000}$ & $0.9160$ & $0.9287$\tabularnewline
\hline
$D_{5}$ & $0.8863$ & $0.9285$ & $0.8819$ & $0.9162$ & $\boldsymbol{1.0000}$ & $0.9169$\tabularnewline
\hline
$D_{6}$ & $0.9279$ & $0.9030$ & $0.9256$ & $0.9141$ & $0.9171$ & $\boldsymbol{1.0000}$\tabularnewline
\hline
\end{tabular}\hfill{}

\caption{\label{tab:gesture-test-medium-noise}Shape determination test using $J_{s}$  for elastic media with noise $5\%$.}
\end{table}

\section{Conclusion}\label{sect:conclusion}
We proposed and analyzed a reconstruction scheme with elastic wave detection for the nearly incompressible materials. The inverse scattering theory of elastic wave is of essential importance.  We employed the translation relations between the scattering by incident elastic point source and the scattering by incident elastic plane waves. We also analysed the asymptotic amplitude of scattering shear wave and pressure wave. Besides, we presented a detailed low frequency analysis for the nearly incompressible materials. With these theoretical analysis, we proposed a two-stage reconstruction algorithm. In the first stage, we employ low frequency scattering data to locate the positions with special designed test functions, and in the second stage we employ the regular frequency scattering data to determine the shapes. We use a precomputed dictionary to store the data needed for the reconstruction algorithm, and the involved computations are mainly inner products of the corresponding data, which are very robust to noise. Various numerical experiments also show the efficiency of the proposed algorithm.

\section*{Acknowledgement}
{\small
The work of J. Li was supported by the NSF of China under the grants No.\, 11571161 and 11731006, the Shenzhen Sci-Tech Fund No. JCYJ20160530184212170 and the SUSTech startup fund.
The work of H. Liu was supported by the FRG grants and startup fund from Hong Kong Baptist University, and Hong Kong RGC General Research Funds (12302415 and 12302017).
H. Sun acknowledges the support of
Fundamental Research Funds for the Central Universities, and the
research funds of Renmin University of China (15XNLF20) and NSF of China under grant No. \,11701563.
}

\bibliographystyle{plain}

\begin{thebibliography}{99}
\bibitem{AR} {K. Aki, P. G. Richards},  {\it Quantitative Seismology}, University Science Books, Sausalito, California, Second Edition, 2002.

\bibitem{BLP} {L. Beir\~{a}o da Veiga, C. Lovadina, L. F. Pavarino}, {\it Positive definite balancing Neumann-Neumann preconditioners for nearly incompressible elasticity}, Numer. Mathematik, 104, 2006, pp. 271--296.

\bibitem{AA1} {H. Ammari, T. Boulier, J. Garnier}, {\it Modeling active electrolocation in weakly
electric fish}, SIAM J. Imaging Sci., 6 (2013), pp. 285--321.

\bibitem{AA2} {H. Ammari, T. Boulier, J. Garnier, W. Jing, H. Kang, H. Wang}, {\it Target identification using dictionary matching of generalized polarization tensors}, Found. Comput.
Math., 14(2014), pp. 27--62.

\bibitem{AA3} {H. Ammari, M. Tran, H. Wang}, {\it Shape identification and classification in echolocation}, SIAM J. Imaging Sci., 7(3), (2014), pp. 1883--1905.


\bibitem{AK} {C. J. S. Alves, R. Kress}, {\it On the far-field operator in elastic obstacle scattering}, IMA J. Appl. Math.,
67(1):1--21, 2002.

\bibitem{BM} {D. Braess, P. Ming}, {\it A finite element method for nearly incompressible elasticity problems}, Mathematics of Computation, 74(249), pp. 25--52, 2004.


\bibitem{CK0} {D. Colton, R. Kress}, {\it Integral Equation Methods in Scattering Theory}, New York: Wiley-Interscience, 1983.

\bibitem{CK} {D. Colton, R. Kress}, {\it Inverse Acoustic and Electromagnetic Scattering Theory}, Springer, Applied Mathematical Sciences Vol.93, Third Edition, 2013.

\bibitem{DR} {G. Dassios, R. Kleinman},  {\it Low Frequency Scattering}, Clarendon Press, Oxford, 2000.

\bibitem{PH0} {P. H\"{a}hner}, {\it A uniqueness theorem in inverse scattering of elastic waves}, IMA J. Appl. Math., 51, pp. 201--215, 1993.

\bibitem{PHhab} {P. H\"{a}hner}, {\it On acoustic, electromagnetic, elastic scattering problems in inhomogeneous media}, Habilitation thesis, 1998.

\bibitem{PH} {P. H\"{a}hner,  G. C. Hsiao}, {\it Uniqueness theorems in inverse obstacle scattering of elastic waves}, Inverse Problems, 9 (1993) pp. 525--534.


\bibitem{HLLS} {G. Hu, J. Li, H. Liu, H. Sun}, {\it Inverse elastic scattering for multiscale rigid bodies with a single far-field
pattern}, SIAM J. Imaging Sci., 7(3), pp. 1799--1825, 2014.



\bibitem{KK} {K. Kiriaki}, {\it A unique solvable integral equation for the Neumann problem in linear elasticity}, Applicable Analysis, Vol. 73(3--4), pp. 379--392, 1999.

\bibitem{Kir1} { A. Kirsch, F. Hettlich}, {\it The Mathematical Theory of Time-Harmonic Maxwell's Equations
Expansion-, Integral-, and Variational Methods}, Applied Mathematical Sciences, Vol. 190, Springer 2015.


\bibitem{Klib1} {M. V. Klibanov}, {\it On the first solution of a long standing problem: Uniqueness of the phaseless quantum inverse scattering problem in 3-d}, Applied Mathematics Letters,
Vol. 37, 2014, pp. 82--85.

\bibitem{Klib2} {M. V. Klibanov, V. G. Romanov}, {\it Two reconstruction procedures for a 3D phaseless inverse scattering problem for the generalized Helmholtz equation},   Inverse Problems, Vol. 32, No. 1, 015005, 2016.


\bibitem{KV} {H. Kobayashia, R. Vanderby}, {\it Acoustoelastic analysis of reflected waves in nearly incompressible, hyper-elastic materials: Forward and inverse problems}, J. Acoust. Soc. Am., 121, pp. 879--887, 2007.

\bibitem{Kup} {V. D. Kupradze}, {\it Three-dimensional Problems of the Mathematical Theory of Elasticity and Thermoelasticity}, North-Holland, 1979.


\bibitem{LAU}  {L. D. Landau, L. P. Pitaevskii, A. M. Kosevich, E. M. Lifshitz}, {\it Theory of Elasticity}, Third Edition: Vol. 7 (Course of Theoretical Physics), Butterworth-Heinemann, Oxford, 1986.

\bibitem{LLS} {J. Li, H. Liu, H. Sun}, {\it On a gesture-computing technique using electromagnetic waves}, \url{http://arxiv.org/abs/1708.02848}, 2017.

\bibitem{LWY} {H. Liu, Y. Wang, C. Yang}, {\it Mathematical design of a novel gesture-based instruction/input device using wave detection}, SIAM J. Imaging Sci., 9(2) (2016), pp. 822--841.

\bibitem{MDR} {P. H. Mott, J. R. Dorgan, C. M. Roland}, {\it The bulk modulus and Poisson's ratio
of ``incompressible" materials}, Journal of Sound and Vibration, 312, pp. 572--575, 2008.

\bibitem{MRR} {P. H. Mott, C. M. Roland, R. D. Corsaro}, {\it Acoustic and dynamic mechanical properties of a polyurethane rubber}, J. Acoust. Soc. Am. 111 (4), pp. 1782--1790, 2002.

\bibitem{Sini} {D. Prasad, M. Sini}, {\it The Foldy-Lax approximation of the scattered waves by many small bodies for the Lame system}, Mathematische Nachrichten, 288(16), pp. 1834--1872 (2015).


\bibitem{SMG} {K. R. Srinivasan, K. Matou\v{s}, P. H. Geubelle}, {\it Generalized finite element method for modeling nearly
incompressible bimaterial hyperelastic solids}, Comput. Methods Appl. Mech. Engrg., 197, pp. 4882--4893, 2008.

\bibitem{SBC} {B. A. Szab\'{o}, I. Babu\v{s}ka, B. K. Chayapathy}, {\it Stress computations for nearly incompressible materials by the p-version of the finite element method}, Int. J. Numer. Meth. Engng., 28(9), pp. 2175--2190, 1989.

\bibitem{YT}  {A. F. Yee, M. T. Takemori}, {\it Dynamic bulk and shear relaxation in glassy polymers. I. Experimental techniques and results on PMMA}, Journal of Polymer Science: Polymer Physics Edition, 20, pp. 205--224, 1982.



\end{thebibliography}

\end{document}